\documentclass[a4paper,twoside,11pt,reqno]{amsart}
\usepackage{amssymb}
\usepackage{latexsym}
\usepackage{amsmath}
\usepackage{color}
\usepackage{graphicx}
\usepackage{subcaption}
\date{}

\begin{document}
	\let\ds\displaystyle
	
	\newtheorem{Th}{Theorem}[section]
	\newtheorem{Le}[Th]{Lemma}
	\newtheorem{Cor}[Th]{Corollary}
	\newtheorem{example}[Th]{Example}

	\newtheorem{rem}{Remark}[section]

	
	
	
	\newcommand{\salto}[1]{[\![#1]\!]_\xi}

	\newcommand{\set}[1]{\left\lbrace #1\right\rbrace}
	\providecommand{\abs}[1]{\left\lvert#1\right\rvert}
	\providecommand{\norm}[1]{\left\lVert#1\right\rVert}
	\newcommand{\bj}[1]{\left( #1\right)}
	\newcommand{\qtq}[1]{\quad \text{#1}\quad }
	\newcommand{\red}[1]{\color{red} #1\color{black}}
	\newcommand{\blue}[1]{\color{blue} #1\color{black}}
	\def\ZZ{\mathbb Z}

	\graphicspath{{figures/}}
	
	
	\setlength{\baselineskip}{1.3\baselineskip}

\title[The semilinear suspension bridge model]{The semilinear suspension bridge model with pointwise dissipation
}

\author[V. Komornik]{Vilmos Komornik}
\address{Département de mathématique, Université de Strasbourg
}%
\email{vilmos.komornik@math.unistra.fr}%

\author[J.E. Mu\~{n}oz Rivera]{Jaime E. Mu\~{n}oz Rivera}%
\address{Universidad del B\'\i o B\'\i o Collao 1202, Casilla 5C, \newline
\indent Concepci\'{o}n - Chile}

\address{Laboratório Nacional de Computação Científica (LNCC) , \newline
\indent Petr\'{o}polis - Brasil}
\email{jemunozrivera@gmail.com}%


\begin{abstract}

This paper is devoted to the analysis of a semilinear suspension bridge model with pointwise localized dissipation. The main contribution of the work is the development of a robust semigroup framework that substantially simplifies the stability analysis of coupled systems with localized damping.

The core of the approach relies on a comparison principle involving the essential growth bound of the associated semigroups together with the invariance of the essential spectrum under compact perturbations. This methodology allows us to reduce the stability analysis of the full coupled system to that of suitably chosen auxiliary systems, thereby avoiding the delicate multiplier techniques and observability arguments commonly used in the literature.

As a first consequence, we establish the exponential stability of the corresponding linear homogeneous system. The semilinear problem is then treated as a locally Lipschitz perturbation of the exponentially stable linear semigroup, yielding global well-posedness and exponential decay of solutions.

Finally, we investigate the long-time dynamics of the semilinear system. By combining the dissipative structure of the model with suitable regularization properties of the nonlinear semigroup, we prove the existence of a global compact attractor and an exponential attractor with finite fractal dimension. These results provide a complete description of the asymptotic behavior of solutions and show that the proposed framework is sufficiently flexible to be adapted to other coupled systems with localized dissipation.

\end{abstract}

\vspace*{5mm} \noindent {\small \underline{2020 Mathematics Subject
			Classification}: 35Q74, 35B41, 74K10 }
\renewcommand{\theequation}{\thesection.\arabic{equation}}

\maketitle

\section{Introduction to Suspension Bridge Dynamics}\label{section-1}

The stabilization of oscillations plays a critical role in the study of suspension bridge dynamics, as it directly affects structural integrity. 
Analyzing the global dynamics of a suspension bridge system is essential for practical applications but poses significant challenges, particularly when considering the nonlinear coupling between the roadbed and the main suspension cable over the spatial domain $x\in (0,\ell)$, where $\ell$ is the span length. 
In recent years, the study of nonlinear vibrations in suspension bridges has attracted increasing attention, resulting in a substantial body of literature. 
While many works focus on approximations and numerical simulations, others employ analytical methods to investigate periodic or long-term global dynamics, building on the pioneering Lazer--McKenna approach \cite{lazer-mckenna}.

The Lazer--McKenna model \cite{lazer-mckenna} introduced a nonlinear dynamical framework for vertical vibrations in suspension bridges, where nonlinearity arises from the supporting cable stays (suspenders). 
These act as one-sided springs, restraining downward motion of the roadbed but not upward motion, leading to a ``jumping nonlinearity'' in the system.

In most studies, the roadbed is modeled using the Euler--Bernoulli beam theory, the main suspension cable is treated as a taut elastic string, and the suspenders are represented as a continuously distributed vertical load. 
Recent works \cite{ref8, ref9, ref10, ref11, ref12, ref13, holubova-janousek, lazer-mckenna} have generalized the Lazer--McKenna model by incorporating midplane stretching due to roadbed elongation, resulting in a doubly nonlinear system. 
This nonlinearity arises from two primary sources:
\begin{enumerate}
\item The flexible cable stays, which behave elastically as one-sided springs, introducing a nonlinearity due to their tension-only behavior.
\item Geometric nonlinearity in the bending equation, resulting from the elongation of the roadbed during vibration.
\end{enumerate}

Consequently, whether linear or nonlinear models are considered, most studies (e.g., \cite{ref1, ref2, ref3, ref6, ref15, ref16}) adopt the Euler--Bernoulli framework to model the transverse vibrations of the roadbed. 
Although the Euler--Bernoulli beam theory remains popular for its simplicity and effectiveness, it serves as a fundamental approach that may benefit from further refinements to capture the complex dynamics of suspension bridges.

The model proposed by Lazer and McKenna \cite{lazer-mckenna} describes the nonlinear vertical oscillations of a suspension bridge, focusing on the interaction between the main cable displacement $v(x,t)$ and the bridge deck displacement $u(x,t)$ over the spatial domain $x\in (0,\ell)$, where $\ell$ is the span length, and time $t\geq 0$.
The system captures the nonlinear behavior of suspenders acting as one-sided springs.

\subsection{Governing Equations}
We consider the coupled nonlinear partial differential equations
\begin{align}
\rho_1 v_{tt}+\beta _0 v_{xx}+k_s (v-u)^+ &= 0, \label{eq:cable} \\
\rho_2 u_{tt}+EI u_{xxxx} -\alpha _0u_{xx}+ k_s (u-v)^+ &= f(x,t) \label{eq:deck}
\end{align}
with the following notations:
\begin{itemize}
\item $v(x,t) $: vertical displacement of the main cable;
\item $u(x,t) $: vertical displacement of the bridge deck;
\item $\rho_1 $: linear density of the main cable;
\item $\rho_2 $: linear density of the bridge deck;
\item $\alpha _0,\beta_0 $: horizontal pretension in the cable and suspenders;
\item $EI $: flexural rigidity of the deck;
\item $k_s $: stiffness coefficient of the suspenders;
\item $(v-u)^+=\max(v-u, 0) $: nonlinear suspender force, active when $v > u $;
\item $f(x,t) $: external forcing (e.g., wind or traffic loads).
\end{itemize}
The subscripts $t $ and $x $ denote partial derivatives with respect to time and space, respectively.

\subsection{Boundary Conditions}
They reflect fixed cable ends and a simply supported deck:
\begin{itemize}
\item For the cable:
\begin{equation*}
v(0,t)=v_x(\ell,t)=0.
\end{equation*}
\item For the deck:
\begin{equation*}
u(0,t)=u_{xx}(0,t)=0, \quad u_{x}(\ell,t)=u_{xxx}(\ell,t)=0.
\end{equation*}
\end{itemize}

\subsection{Initial Conditions}
They are the following:
\begin{equation*}
v(x,0)=v_0(x), \quad v_t(x,0)=v_1(x), \quad u(x,0)=u_0(x), \quad u_t(x,0)=u_1(x).
\end{equation*}

\subsection{Physical Interpretation}
The system models the nonlinear dynamics of a suspension bridge:
\begin{itemize}
\item \textbf{Main Cable ($v $)}: Equation \eqref{eq:cable} describes the cable as a taut string with tension $p_0 v_{xx} $ and inertia $\rho_1 v_{tt} $. 
The term $k_s (v-u)^+ $ models the suspender's tension, active only when the cable is above the deck ($v > u $).

\item \textbf{Bridge Deck ($u $)}: Equation \eqref{eq:deck} models the deck as an Euler--Bernoulli beam with bending rigidity $EI u_{xxxx} $, inertia $\rho_2 u_{tt} $, and suspender force $k_s (u-v)^+ $. 
The external force $f(x,t) $ represents wind or traffic loads.

\item \textbf{Nonlinear Coupling}: The one-sided spring term $(v-u)^+ $ introduces a ``jumping nonlinearity' as suspenders go slack when compressed, leading to large-amplitude oscillations or bifurcations.
\end{itemize}

Pointwise damping has been widely investigated in the literature (see, e.g., \cite{1,2,3,Ammari}), mainly in the context of single evolution equations. 
However, the analysis of coupled systems with localized damping is significantly more delicate, and existing approaches often rely on technically demanding arguments.

One of the early contributions in this direction is due to Tucsnak \cite{Tucsnak}, who studied the stabilization of the one-dimensional wave equation with pointwise damping acting at an interior point. 
He proved that the system is strongly stable, but that the decay of the energy is not uniform in the natural energy space. 
In particular, for sufficiently regular initial data, the solutions exhibit at most a polynomial decay rate, which depends on the arithmetic properties of the location of the damping point. 
His approach is based on a detailed spectral analysis of the associated dissipative operator.

Subsequently, Ammari, Tucsnak and Henrot \cite{AmmariTucsnak} obtained a complete characterization of exponential stability for the same model. 
They showed that exponential decay holds if and only if the location of the damping point satisfies a suitable arithmetic condition, and further identified the optimal decay rate, proving that it is maximal when the actuator is located at the midpoint of the domain. 
Their method relies on a time-domain approach based on observability inequalities for the associated conservative system, combined with spectral analysis.

These results were later refined in \cite{AmmariHenrotTucsnak}, where Ammari, Henrot and Tucsnak provided a complete description of the asymptotic behavior of solutions. 
In particular, they established explicit polynomial decay rates in the case where exponential stability fails, showing that the decay strongly depends on the Diophantine properties of the damping location. 
Their analysis combines observability inequalities, spectral techniques, and Riesz basis arguments.

It is worth emphasizing that all the above contributions deal exclusively with linear models.

 The main contribution of the present paper is the introduction of a novel and robust semigroup framework for the analysis of coupled systems with pointwise damping. The central idea is to characterize the exponential stability of the full coupled system through the essential growth bound of suitably chosen auxiliary semigroups together with the invariance of the essential spectrum under compact perturbations.

More precisely, our approach allows us to interpret the coupled dynamics as a compact perturbation of an exponentially stable decoupled system, reducing the stability analysis to the verification of resolvent conditions and the absence of purely imaginary eigenvalues. In contrast with the classical literature, where pointwise damping problems are usually treated through delicate multiplier techniques, observability inequalities, or intricate spectral decompositions, the proposed methodology provides a considerably simpler and more flexible strategy.

Within this framework, we first establish the exponential stability of the linear homogeneous system. We then treat the semilinear problem as a locally Lipschitz perturbation of the exponentially stable semigroup and prove global well-posedness together with exponential decay of solutions. Finally, we analyze the asymptotic dynamics of the nonlinear system and prove the existence of both a global compact attractor and an exponential attractor, thereby obtaining a complete description of the long-time behavior of solutions.

The proposed approach is sufficiently general to be applicable to a broader class of coupled evolution equations with localized dissipation.

The remainder of the paper is organized as follows. In Section \ref{section-2}, we develop the abstract stability framework and prove the exponential stability of the associated homogeneous system under suitable assumptions. Sections \ref{section-3} and \ref{section-4} are devoted to the semilinear problem, where we establish global well-posedness and prove the existence of a global compact attractor, yielding a complete description of the long-time dynamics of the system.
 
\section{Semigroup approach for the homogeneous system}\label{section-2}
	
\setcounter{equation}{0}
	
Note that the semilinear problem
\begin{align*}
\rho_1 v_{tt}+\beta_0 v_{xx} +\gamma v_{t}(\xi,t) \delta_{\xi}+ k (v-u) &= F(u,v), \\
\rho_2 u_{tt}+\alpha  u_{xxxx}-\alpha_0 u_{xx} +\gamma_{0}u_{t}(\xi,t) \delta_\xi  
+k (u-v) &=G(u,v), 
\end{align*}
encompasses several models of physical interest. The nonlinear terms $F$ and $G$ model 
the suspender forces acting between the cable and the deck. The classical 
Lazer--McKenna model \eqref{eq:cable}--\eqref{eq:deck} is recovered by taking
\begin{equation*}
    F(u,v) = -(v-u)^{-} \quad \text{and} \quad G(u,v) = -(u-v)^{-},
\end{equation*}
which represent the one-sided spring behavior of the cable stays, active only when 
the cable is above the deck. The linear homogeneous case is obtained by setting 
$F = G = 0$. Throughout this section, $\alpha = EI$ denotes the flexural rigidity 
of the deck, and the pointwise dissipative mechanism is localized at the interior 
point $\xi \in (0,\ell)$, with damping intensities governed by the two positive 
constants $\gamma$ and $\gamma_0$.

The strategy adopted in this paper is the following,
we first establish exponential stability for the linear 
homogeneous system.
The well-posedness and asymptotic stability of the full semilinear problem then follow
by treating the nonlinear terms as a 
Lipschitz perturbation of the linear semigroup, as carried out in Section 3 and  Section~\ref{section-4}.

For simplicity, we assume henceforth that $\rho_1 = \rho_2 = 1$. The linear homogeneous 
system then reads
\begin{align}
v_{tt}-\beta_0 v_{xx} +\gamma v_{t}(\xi,t) \delta_{\xi}+ k (v-u) &= 0, 
\label{eq:cableL} \\
u_{tt}+\alpha u_{xxxx}-\alpha_0 u_{xx} +\gamma_{0}u_{t}(\xi,t) \delta_\xi 
+k (u-v) &= 0. \label{eq:deckL}
\end{align}

In order to apply the semigroup theory, first we transform it into a transmission problem. 
Set $I=(0,\xi)\cup (\xi,\ell)$, and rewrite \eqref{eq:cableL}--\eqref{eq:deckL} in the form
\begin{align}
v_{tt}-\beta _0 v_{xx} +k (v-u) &= 0, \quad (x,t)\in I \times \mathbb{R}^{+}_{0}, \label{eq:cableT} \\
u_{tt} +\alpha  u_{xxxx}-\alpha_0 u_{xx}  +k (u-v) &=0, \quad (x,t)\in I \times \mathbb{R}^{+}_{0},\label{eq:deckT}
\end{align}
with the following transmission conditions:
\begin{align}\label{ct.2}
\salto{\beta_0 v_{x}}=-\gamma v_{t}(\xi,t), \quad \salto{\alpha u_{xxx}}=-\gamma_{0} u_{t}(\xi,t), \quad 
\end{align}
here  $\salto{f}=f(\xi^{+})-f(\xi^{-})$ denotes the jump of $f$ at  $\xi$. 
We consider the boundary conditions
\begin{align}\label{cc.2}
u(0,t)=u_{xx}(0,t)=u_x(\ell,t)=u_{xxx}(\ell,t)=0\text{ and } v(0,t)=v_x(\ell,t)=0\text{ for }t>0,
\end{align}
and the initial conditions
\begin{align}\label{ci.2}
v(x,0)=v_{0}(x), \quad v_{t}=v_{1}(x), \quad u(x,0)=u_{0}(x), \quad u_{t}=u_{1}(x)\text{ for }x\in I\text{ and }t>0.
\end{align}
Then the energy associated with system \eqref{eq:cableT}--\eqref{eq:deckT} is given by the formula
\begin{equation*}
E(t)=\dfrac{1}{2}\int_{0}^{\ell} |u_{t}|^2+\alpha |u_{xx}|^2+\alpha_0 |u_{x}|^2+|v_{t}|^2+\alpha_0 |v_{x}|^2\, dx,
\end{equation*}
and it satisfies the inequality 
\begin{equation*}
\dfrac{d}{dt} E(t)=-\gamma \left| v_{t} (\xi,t)\right| ^{2}- \gamma_{0} \left| u_{t} (\xi,t)\right| ^{2}
\end{equation*}
by a direct computation. 
Set
\begin{equation}\label{DefH}
\mathcal{V}_1=\left \{w\in H^{1}(0,\ell),\,\, w(0)=0\right\},\quad \mathcal{V}_2=\left \{w\in H^{2}(0,\ell),\,\, w(0)=0,\; w_x(\ell)=0\right\}.
\end{equation}
The phase space we consider is 
\begin{equation*}
\mathcal{H}:=\mathcal{V}_1\times L^2(0,\ell)\times \mathcal{V}_2 \times L^2(0,\ell). 
\end{equation*}

Set $z=v_t$, $w=u_t$, then for any $U=(v,z,u,w)^\top\in \mathcal{H}$ we define 
\begin{equation}\label{H-norm}
\|U\|_{\mathcal{H}}^2
=\dfrac{1}{2}\int_{0}^{\ell} 
\alpha_0 |v_x|^2
+|z|^2
+\alpha |u_{xx}|^2
+\alpha_0 |u_{x}|^2
+|w|^2\, dx.
\end{equation}
Hence $\mathcal{H}$ is a Hilbert space. 
Next we set
\begin{equation*}
\mathcal{H}^2=[H^2(I)\cap \mathcal{V}_1]\times  \mathcal{V}_1\times [H^{4}(I)\cap \mathcal{V}_2] \times \mathcal{V}_2,
\end{equation*}
\color{black}
and we introduce the operator (here $\mathbf{I}$ stands for the identity maps)
\begin{equation}\label{AA}
{\mathcal{A}}=
\begin{pmatrix}
0&\mathbf{I}&0&0\\
-\beta _0\partial_{x}^{2}-k\mathbf{I}& 0&k\mathbf{I}&0\\
0& 0&0&\mathbf{I}\\
k\mathbf{I}& 0&-\alpha \partial_{x}^{4}+\alpha_0\partial_{x}^{2}&0
\end{pmatrix}
\end{equation}
with the domain	 
\begin{equation}\label{domA}
D(\mathcal{A})=\left\lbrace 
U\in \mathcal{H}^2\ :\ U \text{ verifying }\eqref{ct.2}, \eqref{cc.2} 
\right\rbrace.
\end{equation}
A standard argument shows that $D(\mathcal{A})$ is dense in $\mathcal{H}$, and that $\mathcal{A}$ is a dissipative operator satisfying the equality
\begin{equation*}
\mbox{Re }(\mathcal{A}U,U)_{\mathcal{H}}
=-\gamma \left|z (\xi)\right| ^{2}
-\gamma_{0} \left| w (\xi)\right|^{2}.
\end{equation*} 
Taking inner product to the resolvent equation $i\lambda U-\mathcal{A}U=F$ we get 
\begin{equation}\label{Adiss}
\gamma \left|z (\xi)\right| ^{2}
+\gamma_{0} \left| w (\xi)\right|^{2}=\mbox{Re }(U,F)_{\mathcal{H}}
\end{equation} 

Furthermore, $0\in \varrho(\mathcal{A})$. 
Consequently, the following result holds:

\begin{Th}\label{C00}
The operator $ \mathcal{A} $ defined by \eqref{AA}--\eqref{domA} is the infinitesimal generator of a contraction semigroup $\mathcal{T}$.
\end{Th}

The main result of this section is that $\mathcal{A}$ is the infinitesimal generator of an exponentially stable semigroup:

\begin{Th}\label{TeoEx}
Let $\mathcal{A}$ be the operator defined by \eqref{AA}--\eqref{domA}. 
If  $\xi=\dfrac{n}{2m+1}\ell$ for all $n,m\in \mathbb{N}$ such that $n$ and $2m+1$ co-prime, then
$\mathcal{A}$ is the infinitesimal generator of a strongly continuous and exponentially stable semigroup $\{ \mathcal{T}(t) \}_{t \geq 0}$.
\end{Th}
The main tool to show Theorem \ref{TeoEx} is the following characterization

\begin{Th}\label{TNovo}  
Let $S(t)=e^{{\mathbb{A}}t}$ be a $C_0$-semigroup of contractions on a Banach space. 
Then $S(t)$ is exponentially stable if and only if
\begin{equation}\label{hyp}
i\mathbb{R}\subset \varrho(\mathbb{A})
\quad\text{and}\quad
\omega_{ess}(S(t))<0,
\end{equation}
where $\omega_{ess}(S(t))$ is the essential growth bound of the semigroup $S(t)$.
\end{Th}

\begin{proof}
Here we use \cite[Corollary~2.11, p.~258]{Engel} establishing that the type $\omega$ of the semigroup $e^{\mathbb{A}t}$ satisfies
\begin{equation}\label{idxx}
\omega=\max\{\omega_{ess}, \omega_\sigma(\mathbb{A})\},
\end{equation}
where $\omega_\sigma(\mathbb{A})$ is the upper bound of the real part of the spectrum of $\mathbb{A}$. 	
Moreover, for any $c>\omega_{ess}$, the set
\begin{equation*}
\mathcal{I}_c:=\sigma(\mathbb{A})\cap\set{\lambda\in \mathbb{C}: \mbox{Re}\lambda \geq c}
\end{equation*}
is  finite.

Let us assume that \eqref{hyp} holds.
Since the essential type of the semigroup $\omega_{ess}$ is negative, identity \eqref{idxx} states that the type of the semigroup will be negative provided $\omega_\sigma(\mathbb{A})<0$. 

If $\omega_\sigma (\mathbb {A}) \leq \omega_ {ess}$ then we have nothing to prove.
Let us assume that $\omega_\sigma (\mathbb{A})>\omega_ {ess}$.  
From \eqref{hyp} and  the Hille--Yosida Theorem  we have $\overline{\mathbb {C} _ +} \subset \varrho (\mathbb {A})$, hence $\omega_\sigma (\mathbb {A})\leq 0$.  
On the other hand, 
$\mathcal{I}_ {\omega_{ess}+\delta}$ is finite for every $\delta>0$ satisfying  $\omega_ {ess}+\delta <0$ and $\omega_ {ess}+\delta <\omega_ \sigma (\mathbb {A})$. 
Therefore we have 
\begin{equation*}
\omega_\sigma (\mathbb {A})=\sup \mbox{Re }\sigma(\mathbb{A})=\sup \mbox{Re } \mathcal{I}_ {\omega_{ess}+\delta}<0,
\end{equation*} 
and the sufficient condition follows.
 
Conversely, let us assume that the semigroup $S(t)$ is exponentially stable.
Then  it tends to zero in particular, so that $i\mathbb{R}\subset \varrho(\mathbb{A})$ by   \cite[Theorem~1.1]{DuyBatt}. 
Moreover, by \eqref{idxx}, we also have 
\begin{equation*}
\omega_{ess}
\leq \max\set{\omega_{ess}, \omega_\sigma(\mathbb{A})}
=\omega<0.\qedhere
\end{equation*}
\end{proof}
Note that the above characterization is valid for any Banach space.
 
The proof of Theorem \ref{TeoEx} is based on the following procedure:
\begin{itemize}
\item First we show that the imaginary axis is contained in the resolvent set of $\mathcal{A}$.
\item Next we prove the exponential stability for the uncoupled system \eqref{eq:cableT}--\eqref{eq:deckT} when $k=0$, i.e., for the system of equations
\begin{align}
v_{tt}-\beta_0 v_{xx}  &= 0, \quad (x,t)\in I \times \mathbb{R}^{+}_0, \label{eq:cableD} \\
u_{tt} +\alpha u_{xxxx}-\alpha_0 u_{xx}    &=0, \quad (x,t)\in I \times \mathbb{R}^{+}_0\label{eq:deckD}
\end{align}
with the  transmission, boundary and initial conditions \eqref{ct.2}, \eqref{cc.2} and \eqref{ci.2}. 
We denote by $\mathcal{T}_d$ the corresponding semigroup.

\item Then we show that the difference  $\mathcal{T}(t)-\mathcal{T}_d(t)$ of the semigroups is a compact operator. 
This implies that the essential types are equal.

\item Finally,  using Theorem \ref{TNovo} above we conclude that $\mathcal{T}(t)$ is exponentially stable by showing that the essential type is negative.  
\end{itemize}

We begin by showing that 
$i\mathbb{R} \subset \varrho(\mathcal{A})$. 
\begin{Le}\label{estFor3} 
If $\xi=\dfrac{n}{2m+1}\ell$ with co-prime integers $n$ and $2m+1$, then $i\mathbb{R} \subset \varrho(\mathcal{A})$.
\end{Le}

\begin{proof}
Since the domain of  $\mathcal{A}$  
is compactly embedded in the phase space $\mathcal{H}$ ,and since $0\in \varrho(\mathcal{A})$,  $\mathcal{A}$ is a compact operator.
Hence  
it is enough to show that there are no imaginary eigenvalues. 
Assume on the contrary that there exists an imaginary eigenvalue $i\lambda$ and a corresponding nonzero eigenvector $U=(v,z,u,w)^\top$ such that  $\mathcal{A}U=i\lambda U$. 
In terms of the components we have 
\begin{align}
& i \lambda v -z =0,\label{Aris1} \\
& i \lambda z -\beta_0 v_{xx}+k(v-u) =0,\label{Aris2}\\
& i \lambda u -w =0,\label{Aris3} \\
& i \lambda w +\alpha u_{xxxx}-\alpha_0u_{xx}+k(v-u) =0,\label{Aris4}
\end{align}
satisfying \eqref{ct.2} and \eqref{cc.2}.
They have to be of the form
\begin{equation*}
v_j(x)=B_j\sin\left(\frac{2j+1}{2\ell }\pi x\right),\quad u_j(x)=A_j\sin\left(\frac{2j+1}{2\ell }\pi x\right),\text{ and }z_j=i\lambda v_j,\;\; w_j=i\lambda u_j
\end{equation*}
for some $j\in \mathbb{N}$ and $(A_j, B_j)\ne (0,0)$.
Since $F=0$, from \eqref{Adiss} we get
$z(\xi)=0$, $w(\xi)=0$, so that
\begin{equation*}
\sin\left(\dfrac{2j+1}{2\ell } \pi \xi\right)=0  \Longleftrightarrow  \dfrac{2j+1}{2\ell } \pi \xi=n\pi  \Longleftrightarrow  \xi=\dfrac{2n\ell }{2j+1}
\end{equation*}
for some integer $n$.
This contradicts the assumption imposed on $\xi$, and the lemma follows.
\end{proof}

\subsection{The beam equation}

Our starting point is to show that Equation  \eqref{eq:deckD} with the corresponding  conditions in \eqref{ct.2}, \eqref{cc.2} and \eqref{ci.2} is exponentially stable. 
To show this we set $w=u_{t}$, $U_1=(u,w)^\top$, and we define the phase space 
\begin{equation*}
\mathcal{H}_1:=\mathcal{V}_2 \times L^2(0,\ell)
\end{equation*}
with the norm given by
\begin{equation*}
\|U\|_{\mathcal{H}_1}^{2}:=\int_{0}^{\ell} \alpha |u_{xx}|^2+\alpha_0 |u_{x}|^2+|w|^2\, dx,\quad   U=(u,w)^\top.
\end{equation*}
Furthermore, we introduce the operator 
\begin{equation}\label{do1}
{\mathcal{A}_1}:=
\begin{pmatrix}
0&I\\
-\alpha \partial_{x}^{4}+\alpha_0\partial_{x}^{2}& 0\\
\end{pmatrix}
=
\begin{pmatrix}
0&I\\
-\alpha \partial_{x}^{4}& 0\\
\end{pmatrix}
+
\begin{pmatrix}
0&0\\
\alpha_0\partial_{x}^{2}& 0\\
\end{pmatrix}
=:\mathcal{A}_0+B_0
\end{equation}
with the domains given by
\begin{equation}\label{do2}
D(\mathcal{A}_1)=D(B_0)=D(\mathcal{A}_0):=\left\lbrace U\in \mathcal{H}_1\ :\ w\in \mathcal{V}_2,u_{xx} \in H^{2}(I) ,\; \text{satisfying }\eqref{ct.2} \right\rbrace.
\end{equation}
Then the operators  $\mathcal{A}_0, \mathcal{A}_1$ are dissipative, and
\begin{equation}\label{dis}
\mbox{ Re}(\mathcal{A}_1U,U)=-\gamma_{0}|w(\xi)|^2\le 0.
\end{equation}
The resolvent system is given by 
\begin{equation}\label{resolvente}
(i\lambda I-\mathcal{A}_1)U=F;
\end{equation}
in terms of the coefficients it is written as 
\begin{align}
i\lambda u-w&=f_{1}, \label{r1}\\
i\lambda w+\alpha u_{xxxx}-\alpha_0u_{xx}&=f_{2}, \label{r2}
\end{align}
satisfying 
\begin{equation}\label{ct-t}
\salto{\alpha u_{xxx}}=-\gamma_{0} w(\xi), 
\end{equation}
together with 
\begin{equation}\label{cc-t}
w(0)=w_{xx}(0)=w_x(\ell)=w_{xxx}(\ell)=0.
\end{equation}
Using  \eqref{dis} and \eqref{resolvente} we arrive at
\begin{equation}\label{dis.acotada}
\gamma_{0}|w(\xi)|^2=\mbox{Re }(U,F)_{\mathcal{H}}. 	
\end{equation}

Let $\mathcal{T}_0$ denote the semigroup associated with the model \eqref{eq:deckD} for $\alpha_0=0$. 
In \cite{Ammari}, the exponential stability of the $C_0$-semigroup $\mathcal{T}_0$ was established. 
Using perturbation theory, it is straightforward to show that the corresponding semigroup $\mathcal{T}(t)=e^{\mathcal{A}_1t}$ is also exponentially stable provided that $\alpha_0$ is sufficiently small. 
However, this criterion is not practical, as the required smallness of $\alpha_0$ is relative to the growth bound of the unperturbed semigroup, which is typically unknown and may itself be small. 
We now introduce a new approach that overcomes the limitations of the classical perturbation criterion
and enables us to prove the exponential stability of the semigroup $\mathcal{T}$ for any sufficiently large positive constant $\alpha_0$.

To this end, we first present the following preliminary theorems and lemmas.

\begin{Le}\label{eqDA}
Let $\mathbb{A}: D(\mathbb{A}) \subset \mathbb{H} \to \mathbb{H}$ be the infinitesimal generator of a strongly continuous semigroup $\set{ \mathbb{T}(t)}_{t \geq 0}$ on a  Banach  space $\mathbb{H}$. 
Define the graph norm on $D(\mathbb{A})$ by
\begin{equation*}
\|x\|_{D(\mathbb{A})}=\|x\|_\mathbb{H}+\|\mathbb{A} x\|_\mathbb{H}.
\end{equation*}
Then, the following statements hold:
\begin{enumerate}
\item[(i)] If $\set{ \mathbb{T}(t)}$ is exponentially stable in $\mathbb{H}$, i.e., there exist constants $M, \omega > 0$ such that
\begin{equation*}
\|\mathbb{T}(t)\|_{\mathcal{B}(\mathbb{H})} \leq M e^{-\omega t}, \quad t \geq 0,
\end{equation*}
then $\set{ \mathbb{T}(t)}$ is exponentially stable in $D(\mathbb{A})$ with the same constants $M$ and $\omega$, i.e.,
\begin{equation*}
\|\mathbb{T}(t)\|_{\mathcal{B}(D(\mathbb{A}))} \leq M e^{-\omega t}, \quad t \geq 0.
\end{equation*}

\item[(ii)] If, additionally, $0\in\rho(\mathbb{A})$ (i.e., $\mathbb{A}^{-1} \in \mathcal{B}(\mathbb{H})$) and $\set{ \mathbb{T}(t)}$ is exponentially stable in $D(\mathbb{A})$, then it is also exponentially stable in $\mathbb{H}$. 
Specifically, if
\begin{equation*}
\|\mathbb{T}(t)\|_{\mathcal{B}(D(\mathbb{A}))} \leq M e^{-\omega t}, \quad t \geq 0,
\end{equation*}
then
\begin{equation*}
\|\mathbb{T}(t)\|_{\mathcal{B}(H)} \leq C e^{-\omega t} \text{ with } C=M (1+\|\mathbb{A}^{-1}\|_{\mathcal{B}(\mathbb{H})}), \quad t \geq 0.
\end{equation*}
\end{enumerate}
\end{Le}

\begin{proof}
(i) Let us assume that $\mathbb{T}$ is exponentially stable on $\mathbb{H}$. 
Since $\mathbb{T}(t) D(\mathbb{A}) \subset D(\mathbb{A})$ and $\mathbb{A} \mathbb{T}(t)x=\mathbb{T}(t) \mathbb{A}x$ for all $x\in D(\mathbb{A})$, the graph norm of $\mathbb{T}(t)x$ is
\begin{equation*}
\|\mathbb{T}(t)x\|_{D(\mathbb{A})}=\|\mathbb{T}(t)x\|_\mathbb{H}+\|\mathbb{A} \mathbb{T}(t)x\|_\mathbb{H}=\|\mathbb{T}(t)x\|_\mathbb{H}+ \|\mathbb{T}(t) \mathbb{A} x\|_\mathbb{H}.
\end{equation*}
Using  $\|\mathbb{T}(t)x\|_\mathbb{H} \leq M e^{-\omega t} \|x\|_\mathbb{H}$, we get
\begin{equation*}
\|\mathbb{T}(t)x\|_{D(\mathbb{A})} \leq M e^{-\omega t} \|x\|_\mathbb{H}+M e^{-\omega t} \|\mathbb{A} x\|_\mathbb{H}=M e^{-\omega t} \|x\|_{D(\mathbb{A})},
\end{equation*}
so that $\set{\mathbb{T}(t)}$ is exponentially stable in $D(\mathbb{A})$. 
\medskip

(ii) Since $0\in\rho(\mathbb{A})$,  $\mathbb{A}^{-1}\in\mathcal{B}(\mathbb{H})$. 
For any $x\in \mathbb{H}$, there exists $y=\mathbb{A}^{-1}x\in D(\mathbb{A})$ such that $x=\mathbb{A} y$. 
Using the commutativity of $\mathbb{A}$ and $\mathbb{T}(t)$ on $D(\mathbb{A})$, we have
\begin{equation*}
\mathbb{T}(t) x=\mathbb{T}(t) \mathbb{A} y=\mathbb{A} \mathbb{T}(t) y.
\end{equation*}
Since $\set{T(t)}$ is exponentially stable in $D(\mathbb{A})$, it follows that
\begin{equation*}
\|\mathbb{T}(t) x\|_\mathbb{H}=\|\mathbb{A} \mathbb{T}(t) y\|_\mathbb{H} \leq \|\mathbb{T}(t) y\|_{D(\mathbb{A})} \leq M e^{-\omega t} \|y\|_{D(\mathbb{A})}.
\end{equation*}
Now, compute the graph norm of $y$:
\begin{equation*}
\|y\|_{D(\mathbb{A})}=\|\mathbb{A}^{-1} x\|_\mathbb{H}+\|\mathbb{A} y\|_\mathbb{H}=\|\mathbb{A}^{-1} x\|_\mathbb{H}+\|x\|_\mathbb{H} \leq (1+\|\mathbb{A}^{-1}\|_{\mathcal{B}(\mathbb{H})}) \|x\|_\mathbb{H}.
\end{equation*}
Combining these estimates, we obtain that
\begin{equation*}
\|\mathbb{T}(t) x\|_\mathbb{H} \leq M (1+\|\mathbb{A}^{-1}\|_{\mathcal{B}(\mathbb{H})}) e^{-\omega t} \|x\|_\mathbb{H},
\end{equation*}
proving the exponential stability in $\mathbb{H}$ with the constant $C=M (1+\|A^{-1}\|_{\mathcal{B}(\mathbb{H})})$ and rate $\omega$.
\end{proof}

If $0\notin \rho(\mathbb{A})$, then the inverse $\mathbb{\mathbb{A}}^{-1}$ may not exist, and the above estimate in (ii) cannot be extended to all $x\in H$. 
Thus, the condition $0\in\rho(\mathbb{A})$ is essential for the converse implication.

\begin{Le}\label{estFor} 
If $\xi=\dfrac{n}{2m+1}\ell$ with co-prime integers $n$ and $2m+1$, then $i\mathbb{R} \subset \varrho(\mathcal{A}_1)$.
\end{Le}

\begin{proof}
We proceed as in Lemma \ref{estFor3}.
Since the domain of  $\mathcal{A}_1$ has a compact embedding over the phase space $\mathcal{H}$ and since $0\in \varrho(\mathbb{A})$,  $\mathcal{A}_1$ is a compact operator. 
Hence it is enough to show that there are no imaginary eigenvalues. 
Assume on the contrary that there exists an imaginary eigenvalue $i\lambda$ and a corresponding nonzero eigenvector $U=(u,w)^\top$ such that  $\mathcal{A}_1U=i\lambda U$. 
In terms of the components we have 
\begin{align}
& i \lambda u -w =0,\label{ris1} \\
& i \lambda w +\alpha u_{xxxx}-\alpha_0u_{xx}=0,\label{ris2}
\end{align}
satisfying \eqref{ct-t} and \eqref{cc-t}.
They have to be of the form
\begin{equation*}
u_k(x)=A_k\sin\left(\frac{2k+1}{2\ell}\pi x\right)\text{ and }w_k=i\lambda u_k
\end{equation*}
for  $k\in \mathbb{N}$ and a nonzero constant $A_k$.
Since $F=0$, from \eqref{dis.acotada} we get
$w(\xi)=0$, i.e.,
\begin{equation*}
\sin\left(\dfrac{2k+1}{2\ell } \pi \xi\right)=0  \Longleftrightarrow  \dfrac{2k+1}{2\ell } \pi \xi=n\pi  \Longleftrightarrow  \xi=\dfrac{2n\ell }{2k+1}
\end{equation*}
for some integer $n$.
This contradicts our hypotheses on $\xi$, and the lemma follows.
\end{proof}

We recall the following theorem:
 
\begin{Th}[Ammari--Tucsnak \cite{2}]\label{TAmmari}
For $\alpha_0=0$, the linear system \eqref{eq:deckD}  is exponentially stable in the energy space $\mathcal{V}_2 \times L^2(0,\ell)$ if and only if the location of the damping point $\xi$ relative to the beam length $\ell$ satisfies 
\begin{equation*}
\frac{\xi}{\ell}=\frac{n}{2m+1},
\end{equation*}
 where $n$ and $2m+1$ are coprime positive integers.
\end{Th}

Now we show that Theorem  \ref{TAmmari} remains valid for all $\alpha_0>0$:

\begin{Th}\label{ExpEsta}
For any given $\alpha_0>0$, the semigroup associated with \eqref{eq:deckD} is exponentially stable for every $\xi\in (0,\ell)$ such that 
\begin{equation*}
\frac{\xi}{\ell}=\frac{n}{2m+1},
\end{equation*}
 where $n$ and $2m+1$ are coprime positive integers.
\end{Th}

\begin{proof}
Let us consider the decomposition $\mathcal{A}_1=\mathcal{A}_0+\mathcal{B}$ given in \eqref{do1}--\eqref{do2}.
By Theorem  \ref{TAmmari}  $\mathcal{A}_0$ is the infinitesimal generator of an exponentially stable semigroup $\mathcal{T}_0$, associated with the system \eqref{eq:deckD} for $\alpha_0=0$.
Hence, changing the phase space to $D(\mathcal{A}_0)$,  $\mathcal{T}_0$ is also exponentially stable over $D(\mathcal{A}_0)$.

Moreover, the operator $B$ is also well defined in $D(\mathcal{A}_0)$, i.e., $B:D(\mathcal{A}_0)\to D(\mathcal{A}_0)$.
Indeed, if $U=(u,w)^\top\in D(\mathcal{A}_0)$, then 
\begin{equation*}
u\in H^4(I)\cap \mathcal{V}_2 ,\quad w\in \mathcal{V}_2,
\end{equation*}
and $u_{xx}(0)=u_{xxx}(\ell)=0$; this implies that $B(U)=(0,\alpha_0 u_{xx})^\top \in  D(\mathcal{A}_0)$. 

Let us note that for the new phase space $D(\mathcal{A}_0)$ the domain of the infinitesimal generator $\mathcal{A}_0$ is $D(\mathcal{A}_0^2)$.
Since the immersion of $D(\mathcal{A}_0^2)$ into $D(\mathcal{A}_0)$ is compact, this implies that $B$ is a compact operator over $D(\mathcal{A}_0)$, whence
\begin{equation*}
\omega_{\mathrm{ess}}(\mathcal{T})=\omega_{\mathrm{ess}}(\mathcal{T}_0)<0.
\end{equation*}
We conclude the proof by applying Lemmas \ref{eqDA}, \ref{estFor}   and Theorem \ref{TNovo}.
\end{proof}
\color{black}

\subsection{The wave equation}

We now prove the exponential stability of the decoupled wave equation \eqref{eq:cableD} when $\xi\in (0,\ell)$ is chosen appropriately. To achieve this, we apply the following characterization of the exponential stability for $ C_0 $-semigroups:

\begin{Th}[Pr\"{u}ss,\cite{ref24}]\label{pruss}
Let $S(t)=e^{\mathbb{A}t}$ be a $C_0$-semigroup of contractions on a Hilbert space $\mathcal{H}$. Then $S(t)$ is exponentially stable if and only if the following two conditions hold:
\begin{itemize}
\item $i\mathbb{R} \subset \rho(\mathbb{A})$, where $\rho(\mathbb{A})$ is the resolvent set of $\mathbb{A}$;
\item $\limsup_{|\lambda| \to \infty} \| (i\lambda I-\mathbb{A})^{-1} \|_{\mathcal{B}(\mathcal{H})} < \infty$.
\end{itemize}
\end{Th}

Here we will show the exponential stability of  the solution of the wave equation 
\begin{equation}
v_{tt}-\beta_0 v_{xx}=0, \label{eq:cableD2}
\end{equation}
satisfying the  boundary condition \eqref{cc.2} and the following transmission conditions: 
\begin{equation}\label{trxxs}
\salto{v(\cdot,t)}=0,\quad \salto{\beta_0v_x(\cdot,t)}=-\blue{\gamma}v_t(\xi,t).
\end{equation}
Using the Riemann invariants associated with the system \eqref{eq:cableD2}:  
\begin{equation*}
p=v_t-\sqrt{\beta_0}v_x, 
\quad q=v_t+\sqrt{\beta_0}v_x 
\end{equation*}
we have 
\begin{equation*}
v_t=\frac{q+p}{2}\qtq{and} v_x=\frac{q-p}{2\sqrt{\beta_0}},
\end{equation*}
and the evolution problem can be written in the form
\begin{align}
p_t+k_1p_x&=0,\label{eqq1}\\
q_t-k_1q_x&=0,\label{eqq2}
\end{align}
with $k_1=\sqrt{{\beta_0}}$, satisfying the following boundary conditions
\begin{equation}\label{bbco}
q(0)+p(0)=0,\quad q(\ell )-p(\ell )=0.
\end{equation}
Using the definition of $p$ and $q$ the transmission conditions  \eqref{trxxs} can be written as 
\begin{align}
q(\xi^-)+p(\xi^-)&=q(\xi^+)+p(\xi^+),\label{trxxsdc}\\
{q(\xi^+)} -q(\xi^-)+p(\xi^-) -p(\xi^+)
&=-\frac{\gamma}{k_1}(p(\xi^+)+q(\xi^+)).\label{trxxsdc2}
\end{align}
Introducing the notations
\begin{equation*}
\mathbf{K}=\begin{pmatrix}
k_1&0\\
0&-k_1
\end{pmatrix},
\quad \mathfrak{U}=\begin{pmatrix}
p\\
q
\end{pmatrix},\quad F=\begin{pmatrix}
f_1\\
f_2
\end{pmatrix}
,
\end{equation*}
the  system \eqref{eqq1}--\eqref{trxxsdc2} can be written as
\begin{equation}\label{psis}
\mathfrak{U}_t+\mathbf{K}\mathfrak{U}_x=0,\quad \mathfrak{U}(0)=\mathfrak{U}_0\, .
\end{equation}
Let us consider the phase space  $\mathbf{H}_2:=[L^2(0,{\ell })]^2$, and let us define the operator 
\begin{equation*}
\mathbf{A}\mathfrak{U}=  \mathbf{K}\mathfrak{U}_x
\end{equation*}
with the domain
\begin{equation*}
D(\mathbf{A}):=\left \{\begin{pmatrix}
p\\
q\end{pmatrix}\in \mathbf{H}_2:\;\; p,q\in H^1(\mathbf{I})
\text{ satisfying }\eqref{bbco}, \eqref{trxxsdc}, \eqref{trxxsdc2}
\right\}.
\end{equation*}
A standard argument shows that $\mathbf{A}$ is the infinitesimal generator of the semigroup   $\mathcal{T}_2(t)$ associated with \eqref{psis} over $\mathbf{H}_2$. 
Note that the system \eqref{eq:cableD2}--\eqref{trxxs} is equivalent to \eqref{psis}.

The resolvent system is given by
\begin{equation}\label{res0}
i\lambda \mathfrak{U}+ \mathbf{K}\mathfrak{U}_{x}=F.
\end{equation}
Since $\mathbf{K}^2=k_1^2I$, the above system can be rewritten as
\begin{equation}\label{res}
k_1^2\mathfrak{U}_{x}+i\lambda \mathbf{K}\mathfrak{U}=\mathbf{K}F,
\end{equation}
In terms of the components the above system can be written as
\begin{equation*}
p_{x}+\frac{i\lambda }{k_1}p=\frac{1}{k_1}f_1, \quad  q_{x}-\frac{i\lambda }{k_1}q=-\frac{1}{k_1}f_2,  \label{eqdre1}
\end{equation*}
satisfying the boundary conditions \eqref{bbco} and the transmission conditions \eqref{trxxsdc}--\eqref{trxxsdc2}.

\begin{Le}\label{dissi}
The operator $\mathbf{A}$  is dissipative over the phase space $\mathbf{H}_2$.
\end{Le}

\begin{proof}
Note that
\begin{align*}
\left(
\begin{pmatrix}
p\\
q
\end{pmatrix}
,
\mathbf{A}
\begin{pmatrix}
p\\
q
\end{pmatrix}
\right)_{\mathbf{H}_2}
&=
\left(
\begin{pmatrix}
p\\
q
\end{pmatrix}
,
\begin{pmatrix}
k_1 p_x\\
-k_1 q_x
\end{pmatrix}
\right)_{\mathbf{H}_2}\\
&=
\int_0^\xi \left( k_1 p_xp-k_1 q_xq\right)\, dx+\int_\xi^{\ell } \left(k_1 p_xp-k_1 q_xq\right)\, dx\\
&=
\frac {k_1} 2\left(|p(\xi^-)|^2-|p(0)|^2-|q(\xi^-)|^2+|q(0)|^2\right)\\
&\qquad +\frac {k_1} 2\left(|p({\ell })|^2-|p(\xi^+)|^2-|q({\ell })|^2+|q(\xi^+)|^2\right).
\end{align*}
Using the boundary conditions \eqref{bbco}, the transmission condition \eqref{trxxsdc} and the continuity of the sum $p+q$ at $\xi$, we get
\begin{align*}
\left(
\begin{pmatrix}
p\\
q
\end{pmatrix}
, \mathbf{A}
\begin{pmatrix}
p\\
q
\end{pmatrix}
\right)_{\mathbf{H}_2}
&=\frac {k_1} 2\left(|p(\xi^-)|^2-|q(\xi^-)|^2\right)
+\frac{k_1} 2\left(-|p(\xi^+)|^2+|q(\xi^+)|^2\right)\\
&=\frac {k_1} 2\left[\left(p(\xi^-)-q(\xi^-)\right)
+\left(-p(\xi^+)+q(\xi^+)\right)\right]
\left(p(\xi)+q(\xi)\right).
\end{align*}
Finally, using the transmission condition  \eqref{trxxsdc2} we conclude that
\begin{equation}\label{Vdiss}
\left(\begin{pmatrix}
p\\
q
\end{pmatrix}
, \mathbf{A}
\begin{pmatrix}
p\\
q
\end{pmatrix}
\right)_{\mathbf{H}_2}
=-\frac12\gamma \left|p(\xi)+q(\xi)\right|^2,
\end{equation}
and our conclusion follows.
\end{proof}

\begin{Le}\label{estFor3b} 
If $\xi=\dfrac{n}{2m+1}\ell$ with co-prime integers $n$ and $2m+1$, then $i\mathbb{R} \subset \varrho(\mathbf{A})$.
\end{Le}

\begin{proof}
First note that $0\in \varrho(\mathbf{A})$. 
Indeed, we only need to prove that there exists a solution  of the initial value problem \eqref{res} that satisfies the boundary condition. 
So taking $\lambda=0$ we have that the solutions of the initial values problem is given by 
\begin{equation*}
p(x)=p(0)+\frac{1}{k_1}\int_0^xf_1ds,\quad q(x)=q(0)-\frac{1}{k_1}\int_0^xf_2ds.
\end{equation*}
Choosing $q(0)=-p(0)$ we get that $p(0)+q(0)=0$ verifying the first boundary condition. 
With this choice we look for $p(0)$ such that $q(\ell)-p(\ell)=0$.
Since
\begin{equation*}
q(\ell)-p(\ell)=-2p(0)-\frac{1}{k_1}\int_0^\ell f_1+f_2ds,
\end{equation*}
taking $p(0)=-\frac{1}{2k_1}\int_0^\ell f_1+f_2ds$ we always have a solution of the boundary value problem. 

Henceforth we can consider $\lambda\ne 0$. 
Since the resolvent family associated with $\mathbf{A}$ is compact, it is sufficient to show that there are no imaginary eigenvalues. 
Assume on the contrary that $\mathbf{A}W=i\lambda W$ for some $\lambda \in \mathbb{R}$ and $0\ne W=(p,q)^{\top}\in D(\mathbf{A})$; then 
\begin{equation}\label{dispoint}
p(\xi)+q(\xi)=0
\end{equation}
by \eqref{Vdiss}.
The relation $\mathbf{A}W=i\lambda W$ implies that
\begin{equation*}
p_x+i\frac{\lambda}{k_1} p=0\qtq{and}  q_x-i\frac{\lambda}{k_1} q=0.
\end{equation*}
Using the boundary condition $p(0)+q(0)=0$, we solve the above system to obtain 
\begin{equation}\label{formula-of-p-and-q}
p(x)=p(0)e^{-i\frac{\lambda}{k_1}x}\qtq{and} 
q(x)=-p(0)e^{i\frac{\lambda}{k_1}x}.
\end{equation}
Since $W=(p,q)^{\top}\ne 0$, it follows that $p(0)\ne 0$.

Using  the boundary condition $p({\ell})-q({\ell})=0$ we get
\begin{equation*}
p(\ell)-q(\ell)=p(0)\left(e^{-\frac{i\lambda}{k_1}{\ell}}
+e^{\frac{i\lambda}{k_1}{\ell}}\right)=0,
\end{equation*}
whence
\begin{equation}\label{dispoint2}
\lambda=\left(m+\frac 12 \right)\frac{k_1\pi}{\ell}\qtq{for some}m\in\ZZ.
\end{equation}
At the transmission   point $x=\xi $ we get
\begin{equation*}
p(\xi)=p(0)e^{-i\frac{\lambda}{k_1}\xi},\quad q(\xi)=-p(0)e^{i\frac{\lambda}{k_1}\xi},
\end{equation*}
and consequently using \eqref{dispoint} we arrive at 
\begin{equation*}
p(\xi)+q(\xi)=p(0)\left(e^{-i\frac{\lambda}{k_1}\xi}-
e^{i\frac{\lambda}{k_1}\xi}\right)
=0.
\end{equation*}
This implies that $e^{i\frac{2\lambda}{k_1}\xi}=1$, and  hence 
$\frac{2\lambda}{k_1}\xi=2n\pi$.
Substitution of $\lambda$ given in \eqref{dispoint2} yields
\begin{equation*}
\frac{(2m+1)\pi}{2\ell}\xi=n\pi\quad\text{implying}\quad \xi=\frac{2n}{2m+1}{\ell},
\end{equation*}
contradicting our hypothesis on $\xi$.
\end{proof}

Let us introduce the following function:
\begin{equation*}
\mathfrak{F}_{\xi}(\lambda) :=\cos^2\left(\frac{\lambda}{k_1}{\ell }\right)
+\frac{\gamma^2}{k_1^2}\sin^2\left(\frac{\lambda}{k_1}\xi\right)\cos^2\left(
\frac{\lambda}{k_1}({\ell }-\xi)\right).
\end{equation*}

\begin{Le}\label{ineqq}
If $\xi=\dfrac{n}{2m+1}\ell$ with co-prime integers $n$ and $2m+1$, then 
\begin{equation*}
\inf_{\lambda\in\mathbb{R}} \mathfrak{F}_\xi(\lambda) >0.
\end{equation*}
\end{Le}

\begin{proof}
First we observe that the function $\mathfrak{F}_\xi$ is periodic with period $T=2\pi \frac{k_1}{\ell}n$.
Therefore it suffices to prove that
\begin{equation*}
I:=\inf_{\lambda\in[0,T]} \mathfrak{F}_\xi(\lambda) >0.
\end{equation*}

Assume on the contrary that $I=0$.
Then there exists a sequence $(\lambda_n)$ in $[0,T]$ such that
\begin{equation*}
\cos^2\left(\frac{\lambda_n}{k_1}{\ell}\right)
+\frac{\gamma^2}{k_1^2}\sin^2\left(\frac{\lambda_n}{k_1}\xi\right)
\cos^2\left(\frac{\lambda_n}{k_1}({\ell}-\xi)\right)
\to 0.
\end{equation*}
Since the sequence $(\lambda_n)$ is bounded, it has a  subsequence converging to some $\lambda\in[0,T]$, and we deduce from the above relation the equality
\begin{equation*}
\cos^2\left(\frac{\lambda }{k_1}{\ell}\right)
+\frac{\gamma^2}{k_1^2}\sin^2\left(\frac{\lambda }{k_1}\xi\right)
\cos^2\left(\frac{\lambda }{k_1}(\ell-\xi)\right)
=0.
\end{equation*}
Hence we have either
\begin{equation}\label{case1}
\frac{\lambda }{k_1}{\ell}=\frac{2j-1}{2}\pi,\quad \frac{\lambda }{k_1}\xi=\nu\pi,\quad j,\nu\in\mathbb{N},
\end{equation}
or
\begin{equation}\label{case2}
\frac{\lambda }{k_1}{\ell}=\frac{2j-1}{2}\pi,\quad \frac{\lambda }{k_1}(\ell-\xi)=\frac{2\mu-1}{2}\pi,\quad j,\mu\in\mathbb{N}.
\end{equation}
It follows that either
\begin{equation*}
\text{either}\quad \xi=\frac{2\nu}{2j-1}\ell\qtq{or}
\xi=\frac{2(j-\mu)}{2j-1}\ell,
\end{equation*}
and both conclusions contradict our assumption on $\xi$.
\end{proof}

\begin{Th}\label{A1A2}
If $\xi=\dfrac{n}{2m+1}\ell$ with co-prime integers $n$ and $2m+1$, then the semigroup $e^{\mathbf{A}t}$ is exponentially stable  over $\mathbf{H}_2$.
\end{Th}

\begin{proof}
By Theorem \ref{pruss}  and Lemma \ref{estFor3b} it is enough to show that the resolvent operator is uniformly bounded over the imaginary axes.
The solution of the initial value problem \eqref{res} is given in $[0,\xi]$ by 
\begin{align}
p(x)&=p(0)e^{-i\frac{\lambda}{k_1} x}+\frac{1}{k_1}\int_0^xe^{-i\frac{\lambda}{k_1} (x-s)}f_1(s)\;ds,\quad x\in [0,\xi],\label{segp3}\\
q(x)&=-p(0)e^{i\frac{\lambda}{k_1} x}-\frac{1}{k_1}\int_0^xe^{i\frac{\lambda}{k_1} (x-s)}f_2(s)\;ds,\quad x\in [0,\xi].\label{segp4}
\end{align}
To show that the resolvent  problem \eqref{res} has a unique solution for every $F\in \mathbf{H}_2$ is equivalent to show that there exists a unique $p(0)$ such that the solution satisfies $p(\ell )=q(\ell )$. To show that the resolvent operator is uniformly bounded, is equivalent to show that $p(0)$ is uniformly bounded with respect to $F\in \mathbf{H}_2$.

The solution of \eqref{res}  over $[\xi,{\ell }]$ is given by the formulas
\begin{align}
p(x)&=p(\xi^{+})e^{-i\frac{\lambda}{k_1}(x-\xi)}+\frac{1}{k_1}\int_\xi^xe^{-i\frac{\lambda}{k_1} (x-s)}f_1(s)\;ds,\quad x\in [\xi,{\ell }],\label{segp1}\\
q(x)&=q(\xi^{+})e^{i\frac{\lambda}{k_1}(x-\xi)}-\frac{1}{k_1}\int_\xi^xe^{i\frac{\lambda}{k_1} (x-s)}f_2(s)\;ds,\quad x\in [\xi,{\ell }].\label{segp2}
\end{align}
Using \eqref{segp3} and \eqref{segp4} we get
\begin{equation}\label{segp6}
p(\xi^-)=p(0)e^{-\frac{i\lambda}{k_1}\xi}+J_1,\qquad  q(\xi^-)=-p(0)e^{i\frac{\lambda}{k_1}\xi}+J_2,
\end{equation}
where 
\begin{equation*}
J_1=\frac{1}{k_1}\int_0^\xi e^{-\frac{i\lambda}{k_1} (\xi-s)}f_1(s)\;ds,\qquad
J_2=-\frac{1}{k_1}\int_0^\xi e^{\frac{i\lambda}{k_1}(\xi-s)}f_2(s)\;ds.
\end{equation*}
Now we adjust  $q(\xi^+)$ and $p(\xi^+)$ such that the transmission conditions \eqref{trxxsdc2} holds,
\begin{align*}
p(\xi^+) +q(\xi^+)
&=p(\xi^-) +q(\xi^-),\\
p(\xi^+) -q(\xi^+)
&=p(\xi^-)-{q(\xi^-)} +\red{\frac{\gamma}{k_1}}(p(\xi^+)+q(\xi^+)).
\end{align*}
Solving the above system, we get
\begin{equation*}
p(\xi^+)=p(\xi^-)+\frac{\gamma}{2k_1}(p(\xi^-)+q(\xi^-)),\quad q(\xi^+)=q(\xi^-)-\frac{\gamma}{2k_1}(p(\xi^-)+q(\xi^-)).
\end{equation*}
Substitution of  $p(\xi^-)$ and $q(\xi^-)$ given by \eqref{segp6} into the above equations yields 
\begin{align*}
p(\xi^+)&=p(0)e^{-i\frac{\lambda}{k_1}\xi}-\frac{\gamma}{2k_1}p(0)(e^{i\frac{\lambda}{k_1}\xi}-e^{-i\frac{\lambda}{k_1}\xi})+J_1+\frac{\gamma}{2k_1}(J_1+J_2),\\
q(\xi^+)&=-p(0)e^{i\frac{\lambda}{k_1}\xi}+\frac{\gamma}{2k_1}p(0)(e^{i\frac{\lambda}{k_1}\xi}-e^{-i\frac{\lambda}{k_1}\xi})+J_2-\frac{\gamma}{2k_1}(J_1+J_2).
\end{align*}
Hence, with this choice the transmission conditions \eqref{trxxsdc}--\eqref{trxxsdc2} holds.
Finally, multiplying the above equations by $e^{-i\frac{\lambda}{k_1}(x-\xi)}$ and $e^{i\frac{\lambda}{k_1}(x-\xi)}$, respectively, we get
\begin{align*}
p(\xi^{+})e^{-i\frac{\lambda}{k_1}(x-\xi)}
&=p(0)e^{-i\frac{\lambda}{k_1}x}
-\frac{\gamma}{2k_1}p(0)\left(e^{i\frac{\lambda}{k_1}\xi}-e^{-i\frac{\lambda}{k_1}\xi}\right)
e^{-\frac{i\lambda}{k_1}(x-\xi)}\\
&\qquad\qquad +\left(J_1+\frac{\gamma}{2k_1}(J_1+J_2)\right)e^{-\frac{i\lambda}{k_1}(x-\xi)},\\
\noalign{\medskip}
q(\xi^{+})e^{i\frac{\lambda}{k_1}(x-\xi)}
&=-p(0)e^{\frac{i\lambda}{k_1}x}
+\frac{\gamma}{2k_1}p(0)
\left(e^{i\frac{\lambda}{k_1}\xi}-e^{-i\frac{\lambda}{k_1}\xi}\right)
e^{\frac{i\lambda}{k_1}(x-\xi)}\\
&\qquad\qquad +\left(J_2-\frac{\gamma}{2k_1}(
J_1+J_2)\right)e^{\frac{i\lambda}{k_1}(x-\xi)}.
\end{align*}
Using \eqref{segp1}--\eqref{segp2} and $q({\ell })-p({\ell})=0$ we get
\begin{equation*}
0=-p(0)\left(e^{i\frac{\lambda}{k_1}{\ell }}+e^{-i\frac{\lambda}{k_1}{\ell }}\right)
+\frac{\gamma}{2k_1}p(0)(e^{i\frac{\lambda}{k_1}\xi}-e^{-i\frac{\lambda}{k_1}\xi})\left(e^{\frac{i\lambda}{k_1}({\ell }-\xi)}+e^{\frac{-i\lambda}{k_1}({\ell }-\xi)}\right)+G,
\end{equation*}
where 
\begin{align*}
G&=-\left(J_1+\frac{\gamma}{2k_1}(J_1+J_2)\right)e^{-\frac{i\lambda}{k_1}(\ell -\xi)}
+\left(J_2-\frac{\gamma}{2k_1}(
J_1+J_2)\right)e^{\frac{i\lambda}{k_1}(\ell -\xi)}\\
&\qquad -\frac{1}{k_1}\int_\xi^\ell e^{-i\frac{\lambda}{k_1} (\ell -s)}f_1(s)\;ds-\frac{1}{k_1}\int_\xi^\ell e^{i\frac{\lambda}{k_1} (\ell -s)}f_2(s)\;ds.
\end{align*}
For the existence of a solution $U_2\in D(\mathbf{A}_1)$, $p(0)$ has to be chosen such that
\begin{align*}
0&=-2p(0)\cos\left(\frac{\lambda}{k_1}{\ell }\right)
+2i\frac{\gamma}{k_1}p(0)\sin\left(\frac{\lambda}{k_1}\xi\right)
\cos\left(\frac{\lambda}{k_1}({\ell }-\xi)\right)+G\\
&=-2p(0)\left[\cos\left(\frac{\lambda}{k_1}{\ell }\right)
-i\frac{\gamma}{k_1}\sin\left(\frac{\lambda}{k_1}\xi\right)\cos\left(\frac{\lambda}{k_1}({\ell }-\xi)\right)\right]+G.
\end{align*}
This is possible because the coefficient of $p(0)$ is non-zero by Lemma \ref{ineqq}.
We have thus 
\begin{equation*}
2ip(0)=\frac{G}{\cos(\frac{\lambda}{k_1}{\ell })
+i\frac{\gamma}{k_1}\sin(\frac{\lambda}{k_1}\xi)\cos(\frac{\lambda}{k_1}({\ell }-\xi))},
\end{equation*}
and hence
\begin{equation*}
|p(0)|=\frac{|G|}{2\sqrt{\mathfrak{F}_\xi(\lambda)}}
\le c\|F\|_{\mathbf{H}_2}
\end{equation*}
by Lemma \ref{ineqq} with some positive constant $c$. 
It follows that
\begin{equation*}
\|(i\lambda I-\mathbf{A}_1)^{-1}F\|_{\mathbf{H}_2}
=\|U_2\|_{\mathbf{H}_2}
=\left\|
\begin{pmatrix}
p\\
q
\end{pmatrix}
\right\|_{\mathbf{H}_2}
\leq c'\|F\|_{\mathbf{H}_2}
\end{equation*}
with another constant $c'$.
Applying Theorem \ref{pruss} we get the exponential stability of the semigroup  $e^{t\mathbf{A}_1}$.
\end{proof}

\subsection*{Proof of Theorem \ref{TeoEx}}
Thanks to Theorem \ref{ExpEsta} and Theorem \ref{A1A2}  we conclude that the semigroup $\mathcal{T}_d$ given by systems \eqref{eq:cableD2}--\eqref{eq:deckD} is exponentially stable. 
Note that the operator $\mathcal{A}$ given by \eqref{AA} can be written as 
\begin{equation*}
\mathcal{A}=\begin{pmatrix}
\mathcal{A}_1&\mathcal{B}\\
\mathcal{B}&\mathbf{A}
\end{pmatrix},
\quad \mathcal{B}=\begin{pmatrix}
0&0\\
k&0
\end{pmatrix}.
\end{equation*}

It is easy to see that $\mathcal{B}$ is a compact operator over the phase space $\mathcal{H}$.
Using the decomposition
\begin{equation*}
\mathcal{A}=\begin{pmatrix}
\mathcal{A}_1&0\\
0&\mathbf{A}
\end{pmatrix}+\begin{pmatrix}
0&\mathcal{B}\\
\mathcal{B}&0
\end{pmatrix}:= \mathcal{A}_d+\begin{pmatrix}
0&\mathcal{B}\\
\mathcal{B}&0
\end{pmatrix}
\end{equation*}

So we conclude that $\mathcal{A}$ is a compact perturbation of $\mathcal{A}_d$, which implies that 
$\mathcal{T}-\mathcal{T}_d$ is a compact operator over $\mathcal{H}$. 
Therefore using Lemma \ref{estFor} and Theorem \ref{TNovo} we have that $\mathcal{T}$ is exponentially stable. 

\section{General Results}\label{section-3}
\setcounter{equation}{0}
In this section we establish the well-posedness of an abstract semi-linear problem and demonstrate, under appropriate conditions, the existence of a global compact attractor. 

Let $\mathcal{H}$ be a Hilbert space with $(\cdot,\cdot)_{\mathcal{H}}$ and $\|\cdot\|_{\mathcal{H}}$ its inner product and norm, respectively, and let $\mathcal {F}:\mathcal {H}\to\mathcal {H}$ be a  local  Lipschitz function. 
Assume that for every ball $B_R=\set{W\in\mathcal{H}: \|W\|_{\mathcal{H}}\leq R}$, then there exists a  globally  Lipschitz function $\widetilde{\mathcal{F}_R}$ such that
\begin{equation}\label{ff1}
\mathcal{F}(U)=\widetilde{\mathcal{F}_R}(U)\qtq{for every}U\in B_R
\end{equation}
Additionally, we assume that there exists a positive constant $\kappa_0$, and for any $\epsilon>0$ there exists a positive constant $c_\epsilon$ such that
\begin{multline}\label{ff2}
\int_0^t\big(\widetilde{\mathcal{F}_R}(U(s)),U(s)\big)_{\mathcal{H}}\;ds \\
\leq \kappa_0 \|U(0)\|_{\mathcal{H}}^2+\epsilon \|U(t)\|_{\mathcal{H}}^2+c_\epsilon\|\mathcal{F}(0)\|_{\mathcal{H}}^2 \qtq{for every}U\in C([0,T];\mathcal{H}).
\end{multline}
Under these conditions, the following theorem holds:

\begin{Th}\label{LLP1}
Let $ \set{T(t)}_{t \geq 0}$ be an exponentially stable $C_0$-semigroup of contractions on a Hilbert space $\mathcal{H}$, with infinitesimal generator $\mathbb{A}$. 
Let $\mathcal{F} : \mathcal{H} \to \mathcal{H}$ be a locally Lipschitz continuous mapping satisfying \eqref{ff1} and \eqref{ff2}. 
Then, for every initial datum $U_0 \in \mathcal{H}$, there exists a unique global mild solution to the abstract Cauchy problem
\begin{equation}\label{absCE}
U_t-\mathbb{A} U=\mathcal{F}(U), \quad U(0)=U_0.
\end{equation}
Moreover, if $\mathcal{F}(0)=0$, then the solution decays exponentially to zero as $t \to \infty$. 
Finally, if $U_0 \in D(\mathbb{A})$, then the mild solution is in fact a strong solution of \eqref{absCE}.
\end{Th}

\begin{proof} 
By hypotheses, there exist positive constants $c_0$ and $\gamma$ such that $\|T(t)\|\leq c_0e^{-\gamma t}$,
and  $\widetilde{\mathcal{F}_R}$ globally Lipschitz with Lipschitz constant $K_0$ satisfying conditions (\ref{ff1}) and   (\ref{ff2}). 
Let us assume that $U_0\in D(\mathbb{A})$; the general case hence will follow by a standard density argument. 
It is well known that  there exists a unique global mild solution to

\begin{equation}\label{absCE2}
U_{t}^R-\mathbb{A}U^R=\widetilde{\mathcal{F}_R}(U^R),\quad U^R(0)=U_0\in D(\mathbb{A}).
\end{equation}
Since the phase space $\mathcal{H}$ is reflexive,  $U^R$ is a strong solution of \eqref{absCE2} by \cite[p.~189, Theorem 1.6]{Pazy}. 
Hence  $U^R\in L^\infty(0,T;D(\mathbb{A}))$.
Multiplying  equation \eqref{absCE2} by $U^R$ we get that
\begin{equation*}
\frac 12 \frac{d}{dt}\|U^R(t)\|_{\mathcal{H}}^2-(\mathbb{A}U^R, U^R)_{\mathcal{H}}=(\widetilde{\mathcal{F}_R}(U^R),U^R)_{\mathcal{H}}.
\end{equation*}
Since the semigroup is contractive, its infinitesimal generator is dissipative, therefore
\begin{equation*}
\|U^R(t)\|_{\mathcal{H}}^2\leq \|U_0\|_{\mathcal{H}}^2+2\int_0^t(\widetilde{\mathcal{F}_R}(U^R),U^R)_{\mathcal{H}}\;dt.
\end{equation*}
Using \eqref{ff2} for $\epsilon=\frac 14$ we conclude that there exists positive constants such that 
\begin{equation}\label{BoundIn}
\frac 12\|U^R(t)\|_{\mathcal{H}}^2\leq (1+2\kappa_0)\|U_0\|_{\mathcal{H}}^2+c\|\mathcal{F}(0)\|_{\mathcal{H}}^2.
\end{equation}
Note that for $R> 2(1+2\kappa_0)\|U_0\|_{\mathcal{H}}^2+c\|\mathcal{F}(0)\|_{\mathcal{H}}^2$, hypothesis \eqref{ff1} yields 
\begin{equation*}
\widetilde{\mathcal{F}_R}(V)=\mathcal{F}(V),\quad \text{for any $V$ with}\;\; \|V\|_{\mathcal{H}}\leq R.
\end{equation*}
In particular $U^R\in B_R$ so we have 
\begin{equation*}
\widetilde{\mathcal{F}_R}(U^R(t))=\mathcal{F}(U^R(t)).
\end{equation*}
This means that $U^R$ is also solution of system \eqref{absCE} and because of the uniqueness we conclude that $U^R=U$. 
Hence $U$ is the global solution of \eqref{absCE}. 

Now we assume that $\mathcal{F}(0)=0$. To show the exponential stability to system \eqref{absCE}, it is enough to show the exponential decay of the system \eqref{absCE2}. 
To do this, we use a fixed point argument.
Set 
\begin{equation*}
\mathcal{T}(V)(t)=T(t)U_0+\int_0^{t}T(t-s)\widetilde{\mathcal{F}_R}(V(s))\;ds,
\end{equation*}
and introduce the vector space 
\begin{equation*}
E_{\mu}=\left\{V\in L^\infty(0,\infty;\mathcal{H});\;\; t\mapsto e^{\mu t}\|V(t)\|\in L^\infty(0,\infty)\right\}
\end{equation*}
with the complete norm
\begin{equation*}
\norm{V}_{\mu}:=\sup_{t\ge 0}e^{\mu t}\norm{V(t)}_{\mathcal{H}}.
\end{equation*}

Note that $\mathcal{T}$ is invariant over $E_{\gamma-\varepsilon}$ if $0<\varepsilon<\gamma$. 
Indeed, for any $V\in E_{\gamma-\varepsilon}$ and $t\ge 0$ we have
\begin{align*}
\|\mathcal{T}(V)(t)\|_{\mathcal{H}}
&\leq\|U_0\|_{\mathcal{H}}e^{-\gamma t}+\int_0^t\|\widetilde{\mathcal{F}_R}(V(s))\|_{\mathcal{H}}e^{-\gamma(t-s)}\;ds,\\
&\leq\|U_0\|_{\mathcal{H}}e^{-\gamma t}+K_0\int_0^t\|V(s)\|_{\mathcal{H}}e^{-\gamma(t-s)}\;ds,
\end{align*}
and therefore
\begin{align*}
e^{(\gamma-\varepsilon)t}\|\mathcal{T}(V)(t)\|_{\mathcal{H}}
&\le \|U_0\|_{\mathcal{H}}e^{-\varepsilon t}
+K_0e^{-\varepsilon t}\int_0^te^{\varepsilon s}e^{(\gamma-\varepsilon) s}\|V(s)\|_{\mathcal{H}}\; ds\\
&\le \|U_0\|_{\mathcal{H}}e^{-\varepsilon t}
+K_0e^{-\varepsilon t}\int_0^te^{\varepsilon s}\; ds\|V\|_{\gamma-\varepsilon}\\
&\le \|U_0\|_{\mathcal{H}}+\frac{K_0}{\varepsilon}\norm{V}_{\gamma-\varepsilon}.
\end{align*}
Taking the supremum over $t\ge 0$ we conclude that
\begin{equation*}
\norm{\mathcal{T}(V)}_{\gamma-\varepsilon}
\le \|U_0\|_{\mathcal{H}}+\frac{K_0}{\varepsilon}\norm{V}_{\gamma-\varepsilon}<+\infty,
\end{equation*}
so that $\mathcal{T}(V)\in E_{\gamma-\varepsilon}$.

Next we introduce for every fixed $T>0$  the norm
\begin{equation*}
\norm{V}_{\mu,T}:=\sup_{t\in[0,T]}e^{\mu t}\norm{V(t)}_{\mathcal{H}},\quad V\in L^\infty(0,T;\mathcal{H}).
\end{equation*}
Then we  have for all $W_1, W_2\in E_{\gamma-\varepsilon}$, $T>0$ and $t\in[0,T]$ the following estimate:
\begin{align*}
e^{(\gamma-\varepsilon)t}\|\mathcal{T}(W_1-W_2)(t)\|_{\mathcal{H}}
&\le K_0\int_0^te^{\varepsilon (s-t)}e^{(\gamma-\varepsilon) s}\|(W_1-W_2)(s)\|_{\mathcal{H}}\; ds\\
&\le K_0t\|W_1-W_2\|_{\gamma-\varepsilon,T}
\end{align*}
Iterating this, as usual for Volterra equations, we obtain for every positive integer $n$ the estimate
\begin{equation*}
\norm{\mathcal{T}^n(W_1)-\mathcal{T}^n(W_2)}_{\gamma-\varepsilon,T}
\le \frac{(K_0T)^n}{n!}\norm{W_1-W_2}_{\gamma-\varepsilon,T}.
\end{equation*}
This shows that if $n$ is sufficiently large, then $\mathcal{T}^n$ is a contraction on $L^\infty(0,T;\mathcal{H})$ for this norm, so that there exists a unique fixed point of $\mathcal{T}$ in $L^\infty(0,T;\mathcal{H})$.

Letting $T\to+\infty$ we conclude that there exists a unique function $L^\infty(0,\infty;\mathcal{H})$ satisfying
\begin{equation*}
U(t)=T(t)U_0+\int_0^{t}T(t-s)\widetilde{\mathcal{F}_R}(U(s))\;ds 
\end{equation*}
for all $t\ge 0$, i.e., $U$ is a solution of \eqref{absCE2}. Finally, since $\mathcal{T}$ is invariant over $E_{\gamma-\varepsilon}$,  the solution decays exponentially.
\end{proof}

\begin{Cor}\label{CoLLP1}
Let the operator $\mathbb{A}$ satisfy the assumptions of Theorem \ref{LLP1}, and let the nonlinear mapping $\mathcal{F}$ fulfill conditions \eqref{ff1}--\eqref{ff2}. 
Then, for any $R>0$, there exist a constant $t_0=t_0(R)>0$, and positive constants $c_1, c_2$, independent of the initial data, such that the solution of \eqref{absCE} satisfies
\begin{equation}\label{absor}
\|U(t)\|_{\mathcal{H}} \leq c_1 e^{-\mu t} R+c_2 \qtq{for all}t \geq t_0,
\end{equation}
for all initial data $U_0 \in \mathcal{H}$ with $\|U_0\|_{\mathcal{H}} \leq R$.
\end{Cor}

\begin{proof}
Since $\mathcal{F}$ is locally Lipschitz and the semigroup $T(t)=e^{\mathbb{A} t}$ is exponentially stable, there exist constants $c \geq 1$ and $\mu>0$ such that
\begin{equation*}
\|T(t)\|_{\mathcal{L}(\mathcal{H})} \leq c e^{-\mu t}\qtq{for all}t \geq 0.
\end{equation*}
By the variation of constants formula, the solution can be written as
\begin{equation*}
U(t)=T(t)U_0+\int_0^t T(t-s)\mathcal{F}(U(s))\,ds.
\end{equation*}
Taking norms in $\mathcal{H}$ yields
\begin{align*}
\|U(t)\|_{\mathcal{H}}
&\leq \|T(t)U_0\|_{\mathcal{H}}+\int_0^t \|T(t-s)\|_{\mathcal{L}(\mathcal{H})}\,\|\mathcal{F}(U(s))\|_{\mathcal{H}}\,ds \\
&\leq c e^{-\mu t}\|U_0\|_{\mathcal{H}}+c \int_0^t e^{-\mu (t-s)} \|\mathcal{F}(U(s))\|_{\mathcal{H}}\,ds.
\end{align*}

By assumptions \eqref{ff1}--\eqref{ff2}, the nonlinear term $\mathcal{F}$ is bounded on bounded subsets of $\mathcal{H}$. 
In particular, for initial data satisfying $\|U_0\|_{\mathcal{H}} \leq R$, standard arguments (see, e.g., uniform boundedness of trajectories for dissipative systems) ensure that the corresponding solution $U(t)$ remains bounded for $t \geq t_0(R)$. 
Hence there exists a constant $M=M(R)>0$ such that
\begin{equation*}
\|\mathcal{F}(U(t))\|_{\mathcal{H}} \leq M\qtq{for all}t \geq t_0.
\end{equation*}

Therefore, for $ t \geq t_0 $, we obtain
\begin{equation*}
\|U(t)\|_{\mathcal{H}} \leq c e^{-\mu t} R+c M \int_0^t e^{-\mu (t-s)}\,ds
= c e^{-\mu t} R+\frac{c M}{\mu}.
\end{equation*}

Setting $ c_1=c $ and $ c_2=\frac{c M}{\mu} $, the estimate \eqref{absor} follows. 
This shows that the system possesses a bounded absorbing set in $ \mathcal{H} $, completing the proof.
\end{proof}

The nonlinear semigroup $\{S(t)\}_{t \geq 0}$ associated with \eqref{absCE} admits the decomposition 
\begin{equation*}
S(t)=L(t)+N(t),
\end{equation*}
where $S(t)U_0=U(t)$, $L(t)U_0=W(t)$, and $N(t)U_0=V(t)$. 
The component $W(t)$ satisfies
\begin{equation}\label{absCEw1}
\begin{cases}
W_{t}-\mathbb{A}W=\mathcal{F}_0(W), \\
W(0)=U_0 \in \mathcal{H}.
\end{cases}
\end{equation}
Where $ \mathcal{F}_0(W)=\mathcal{F}(W)-\mathcal{F}(0)$. 
According to Theorem \ref{LLP1}, the solution $W$ decays exponentially in $\mathcal{H}$. 
Consequently, the function $V:=U-W$ satisfies the evolution equation
\begin{equation}\label{absCEv}
\begin{cases}
V_{t}-\mathbb{A}V=\mathcal{F}(U)-\mathcal{F}_0(W), \\
V(0)=0,
\end{cases}
\end{equation}
which is globally defined for all $t \geq 0$. 

In order to incorporate the nonlinear effects in the suspension bridge model introduced by Lazer and McKenna \cite{lazer-mckenna}, we assume that the nonlinear term \( \mathcal{F} \) depends only on the displacement components of the state variable \( U \). To facilitate the analysis, we introduce a bounded linear operator
\begin{equation}
\mathfrak{B} : \mathcal{H} \to \mathcal{H},
\end{equation}
which replaces the velocity components of the state vector with the corresponding displacement components, thereby improving the regularity 
properties of the original state \( U \), that is $\mathfrak{B}U_t\in \mathcal{H}$. We assume that
\begin{equation}\label{frakB}
\|\mathfrak{B}U\|_{\mathcal{H}} \leq c\,\|U\|_{\mathcal{H}},
\end{equation}
for some constant \( c>0 \). Consequently, \( \mathfrak{B} \) induces a bounded linear operator on \( C([0,T];\mathcal{H}) \).

\begin{rem}\label{RBBB}
Consider the semigroup generated by the coupled system \eqref{eq:deckL}--\eqref{eq:cableT}, with state variable
\begin{equation}
U=(v,v_t,u,u_t)^\top \in \mathcal{H}.
\end{equation}
Define the operator \( \mathfrak{B} \) by
\begin{equation*}
\mathfrak{B}=
\begin{pmatrix}
0 & 0 & 0 & 0 \\
1 & 0 & 0 & 0 \\
0 & 0 & 0 & 0 \\
0 & 0 & 1 & 0
\end{pmatrix},
\end{equation*}
so that
\begin{equation}
\mathfrak{B}U=(0,v,0,u)^\top.
\end{equation}

By the definition \eqref{H-norm} of the norm in \( \mathcal{H} \), together with the Poincar\'e inequality, condition \eqref{frakB} is satisfied for some positive constant \( c \).

The operator \( \mathfrak{B} \) extracts the displacement components \( (u,v) \) of the state variable, which are precisely the variables through which the nonlinearity \( \mathcal{F} \) acts. In this way, the nonlinear term depends only on components with higher spatial regularity. This observation justifies the identity
\begin{equation}
\frac{d}{dt}\mathcal{F}(U)
=
D\mathcal{F}(\mathfrak{B}U)\,\mathfrak{B}U_t,
\end{equation}
which will be used in the derivation of \eqref{absCEvDf}.
\end{rem}
In addition to assumptions \eqref{ff1}--\eqref{ff2} we assume that $\mathcal{F}$ is continuously differentiable, and has the following properties:
\begin{gather}\label{Freg1}
\mathcal{F}(W)=\widetilde{\mathcal{F}}(\mathfrak{B}W),\quad \|D\mathcal{F}(Y)\|\leq c\|Y\|_{\mathcal{H}}\qtq{for all}Y\in B_R\subset \mathcal{H},\\
\label{Freg2}
\frac{\partial }{\partial t}\mathcal{F}(U),\quad \frac{\partial }{\partial t}\mathcal{F}(W)\in C(0,T;\mathcal{H}),
\end{gather}
where $D$ denotes the derivative of $\mathcal{F}$ and $R>0$. 
Using  \cite[p.~109, Corollary 4.2.11]{Pazy}, we have that the solution of system \eqref{absCEv} satisfies 
\begin{equation*}
V\in C^1(0,T;\mathcal{H})\cap C(0,T;D(\mathcal{A})).
\end{equation*}
Since the initial condition of the system \eqref{absCEv} vanishes, we have:
\begin{equation*}
\lim_{t\rightarrow 0} V(t)=0,\quad \lim_{t\rightarrow 0} \mathcal{F}(U)= \mathcal{F}(U_0),\quad \lim_{t\rightarrow 0} \mathcal{F}_0(W)= \mathcal{F}(U_0)-\mathcal{F}(0) .
\end{equation*}
Using the above limit in \eqref{absCEv} we get that 
\begin{equation*}
\lim_{t\rightarrow 0} V_t(t)=\mathcal{F}(0):=F_0.
\end{equation*}
Differentiating equation \eqref{absCEv} we get 
\begin{equation*}
V_{tt}-\mathbb{A}V_t=D\mathcal{F}(U)\mathfrak{B} U_t-D\mathcal{F}(W) \mathfrak{B}W_t,\quad V_t(0)=F_0\in\mathcal{H},
\end{equation*}
So we have 
\begin{equation}\label{absCEvDf}
V_{tt}-\mathbb{A}V_t=D\mathcal{F}(U)\mathfrak{B}V_t-\left[D\mathcal{F}(W)-D\mathcal{F}(U)\right]\mathfrak{B} W_t,\quad V_t(0)=F_0\in\mathcal{H}.
\end{equation}

\begin{Le}\label{lemaW}
Under the above conditions of $\mathcal{F}$ and assuming additionally that $\mathcal{F}$ is differentiable and $\mathbb{A}$ an infinitesimal generator of exponentially stable semigroup, then  the solution of \eqref{absCEv} satisfies $V\in C(0,T;D(\mathbb{A}))$ and
\begin{equation*}
\|V_t(t)\|_{\mathcal{H}}+
\|V(t)\|_{D(\mathcal{A})}\leq c (\|F_0\|_{\mathcal{H}}+\|U_0\|_{\mathcal{H}})\qtq{for all}t\geq t_0.
\end{equation*}
\end{Le}

\begin{proof}
Let us denote by 
$\mathfrak{W}=V_t$ the solution of 
\begin{equation}\label{SolD}
\mathfrak{W}_{t}-\mathbb{A}\mathfrak{W}=D\mathcal{F}(U)\mathfrak{B}V_t-\left[D\mathcal{F}(W)-D\mathcal{F}(U)\right]\mathfrak{B}W_t\quad \mathfrak{W}(0)=F_0\in\mathcal{H}.
\end{equation}
So we have 
\begin{equation}\label{WWW}
\mathfrak{W}(t)=e^{\mathbb{A}t}F_0+\int_0^te^{\mathbb{A}(t-s)}\left\{D\mathcal{F}(U)\mathfrak{B}V_t-\left[D\mathcal{F}(W)-D\mathcal{F}(U)\right]\mathfrak{B}W_t\right\}ds.
\end{equation}
Using \eqref{frakB} and since $U$ and $W$ are bounded for any $t\geq t_0$ we have:
\begin{align*}
\|D\mathcal{F}(U)\mathfrak{B}V_t\|_{\mathcal{H}}
&\leq c\|U\|_{\mathcal{H}}\|V\|_{\mathcal{H}},\\
\|(D\mathcal{F}(W)-D\mathcal{F}(U))\mathfrak{B}W_t\|_{\mathcal{H}}
&\leq c\|U\|_{\mathcal{H}}\|W\|_{\mathcal{H}}.
\end{align*}
Using Corollary \ref{CoLLP1} and  Theorem \ref{LLP1},  we conclude that 
\begin{equation*}
\|V\|_{\mathcal{H}}\leq c\left(\|U_0\|_{\mathcal{H}}+\|F\|_{\mathcal{H}}\right),\quad \|W\|_{\mathcal{H}}\leq c\|U_0\|_{\mathcal{H}}e^{-\gamma t}.
\end{equation*}
Inserting the above inequalities into \eqref{WWW} and recalling that $\mathbb{A}$ is exponentially stable, we get
\begin{equation}\label{WWW1}
\|\mathfrak{W}(t)\|_{\mathcal{H}}\leq c\|F_0\|_{\mathcal{H}}e^{-\gamma t}+c\int_0^te^{-\gamma (t-s)}\left(\|U\|_{\mathcal{H}}\|V\|_{\mathcal{H}}+\|U\|_{\mathcal{H}}\|W\|_{\mathcal{H}}\right)ds.
\end{equation}
Since $U$, $V$ and $W$ are bounded, we find 
\begin{equation*}
\|\mathfrak{W}(t)\|_{\mathcal{H}}\leq c\|F_0\|_{\mathcal{H}}e^{-\gamma t}+c\int_0^te^{-\gamma (t-s)}\left(\|U_0\|_{\mathcal{H}}+\|F_0\|_{\mathcal{H}}\right)ds.
\end{equation*}
From where we have 
\begin{equation*} 
\|\mathfrak{W}(t)\|_{\mathcal{H}}\leq c(\|U_0\|+\|F_0\|).
\end{equation*}
Using equation \eqref{absCEv} we conclude that 
\begin{equation*}
\|\mathbb{A}V(t)\|_{\mathcal{H}}^2\leq \|F_0\|_{\mathcal{H}}^2+c\|U_0\|_{\mathcal{H}}^2\qtq{for all}t\geq 0.
\end{equation*}
Consequently, our conclusion follows.
\end{proof}

Assuming that $0\in \varrho(\mathbb{A}),$ the operator $\mathbb{A}^{-1}$ is well defined and bounded on $\mathcal H$. 
We then introduce the norm
\begin{equation}
\|U\|_{-1}:=\|\mathbb{A}^{-1}U\|_{\mathcal H},
\end{equation}
and define the extrapolation space
\begin{equation}
\mathcal H_{-1}:=
\overline{\mathcal H}^{\,\|\cdot\|_{-1}}.
\end{equation}
Under these notations, we consider the extrapolated generator
\begin{equation}
\mathbb A_{-1}:D(\mathbb A_{-1})\subset \mathcal H_{-1}
\longrightarrow
\mathcal H_{-1}.
\end{equation}
It is well known that
\begin{equation}
D(\mathbb A_{-1})=\mathcal H.
\end{equation}
We are now in a position to establish the following result.

\begin{Th}
\label{Attractor} 
Let $S(t)$ be the semigroup generated by the abstract
equation~\eqref{absCE}. Suppose that
\begin{enumerate}
\item[\rm(H1)] $0\in\varrho(\mathbb{A})$ and the embedding
  $D(\mathbb{A})\hookrightarrow\mathcal{H}$ is compact.
\item[\rm(H2)] There exists a compact operator
  $\mathfrak{B}:\mathcal{H}\to\mathcal{H}$ such that
  $\mathcal{F}(U)=\mathcal{F}(\mathfrak{B}U)$ for all
  $U\in\mathcal{H}$, and
  \begin{equation}
  \label{eq:K}
      K := \sup_{t\geq 0}
      \left\|\frac{d}{dt}\mathfrak{B}S(t)U_0\right\|_{\mathcal{H}}
      < \infty.
  \end{equation}
\item[\rm(H3)] $\mathcal{F}\in C^1_{\mathrm{loc}}(\mathcal{H};
  \mathcal{H})$ in the Fr\'{e}chet sense.
\end{enumerate}
Then:
\begin{enumerate}
\item[\rm(i)] $S(t)$ possesses a compact global attractor
  $\mathfrak{A}\subset D(\mathbb{A})$ of finite fractal dimension.
\item[\rm(ii)] $S(t)$ possesses a fractal exponential attractor
  $\mathfrak{M}\subset\mathcal{H}_{-1}$ with
  $\mathfrak{A}\subset\mathfrak{M}$ as subsets of
  $\mathcal{H}_{-1}$, positively invariant under $S(t)$, and
  satisfying
  \[
      \mathrm{dist}_{\mathcal{H}_{-1}}(S(t)B,\mathfrak{M})
      \leq C_B\,e^{-\gamma t}, \qquad t\geq 0,
  \]
  for every bounded set $B\subset\mathcal{H}$, where
  $C_B,\gamma>0$. Moreover,
  $\dim_f^{\mathcal{H}_{-1}}\mathfrak{M}<\infty$.
\end{enumerate}
\end{Th}

\begin{rem}
\label{rem:H2-consequence}
Under hypotheses~\rm(H2) and~\rm(H3), the bound~\eqref{eq:K}
implies that $\frac{d}{dt}\mathcal{F}(\mathfrak{B}S(t)U_0)$ is
also uniformly bounded in $\mathcal{H}$. Indeed, by the
Fr\'{e}chet chain rule,
\[
    \frac{d}{dt}\mathcal{F}(\mathfrak{B}S(t)U_0)
    = D\mathcal{F}(\mathfrak{B}S(t)U_0)
      \cdot\frac{d}{dt}\mathfrak{B}S(t)U_0.
\]
Since the orbit $\{\mathfrak{B}S(t)U_0 : t\geq 0\}$ lies in a
bounded ball $B_R\subset\mathcal{H}$ and $\mathcal{F}\in
C^1_{\mathrm{loc}}$, the differential $D\mathcal{F}$ is
uniformly bounded on $B_R$:
$\sup_{B_R}\norm{D\mathcal{F}}_{\mathcal{L}(\mathcal{H})}
\leq C_R < \infty$.
Therefore
\[
    \left\|\frac{d}{dt}\mathcal{F}(\mathfrak{B}S(t)U_0)
    \right\|_{\mathcal{H}}
    \leq C_R\,K < \infty.
    \]
\end{rem}

\begin{proof}
\noindent\textbf{(i).}
Corollary~\ref{CoLLP1} provides the existence of a bounded
absorbing set $\mathbb{B}\subset\mathcal{H}$, while
Lemma~\ref{lemaW} shows that $S(t)\mathbb{B}$ is exponentially
attracted by a bounded set $\mathcal{C}\subset D(\mathbb{A})$.
Since the embedding $D(\mathbb{A})\hookrightarrow\mathcal{H}$ is
compact by~(H1), the set $\mathcal{C}$ is compact in
$\mathcal{H}$. By the standard theory of dynamical systems
\cite{Gatti,MR2271373,Hale,Teman}, this yields a compact global
attractor $\mathfrak{A}\subset\mathcal{C}\subset D(\mathbb{A})$.
To show the finite fractal dimension we will show that the dynamical system $(S(t),\mathcal{H})$ is quasi-stable on any bounded ball. 

\medskip
Indeed let us show the Quasi-stability inequality.
From the Duhamel formula
\begin{equation}
\label{eq:Duhamel}
    S(t)U_0 = T(t)U_0
    + \int_0^t T(t-\sigma)\mathcal{F}(U(\sigma))\,d\sigma
\end{equation}
and the hypothesis $\mathcal{F}(U)=\mathcal{F}(\mathfrak{B}U)$,
the difference of two trajectories satisfies
\begin{equation}
\label{eq:diff}
    \norm{S(t)U_0-S(t)U_1}_{\mathcal{H}}
    \leq \norm{T(t)(U_0-U_1)}_{\mathcal{H}}
    + C\int_0^t \norm{T(t-\sigma)}_{\mathcal{L}(\mathcal{H})}
      \norm{\mathfrak{B}S(\sigma)U_0
            -\mathfrak{B}S(\sigma)U_1}_{\mathcal{H}}
      \,d\sigma.
\end{equation}
Using the exponential stability
$\norm{T(t)}_{\mathcal{L}(\mathcal{H})}\leq Me^{-\omega t}$ and
estimating the convolution in~\eqref{eq:diff} gives
\begin{equation}
\label{eq:qs-pre}
    \norm{S(t^*)U_0-S(t^*)U_1}_{\mathcal{H}}
    \leq Me^{-\omega t^*}\norm{U_0-U_1}_{\mathcal{H}}
    + \frac{MC}{\omega}
      \sup_{t\leq t^*}
      \norm{\mathfrak{B}S(t)U_0-\mathfrak{B}S(t)U_1}_{\mathcal{H}}.
\end{equation}
We choose $t^*>0$ large enough so that
$q:=Me^{-\omega t^*}<1$, and define
\begin{equation}
\label{eq:nH}
    n_{\mathcal{H}}(V)
    := \frac{MC}{\omega}
       \sup_{t\leq t^*}
       \norm{\mathfrak{B}S(t)V}_{\mathcal{H}}.
\end{equation}
Then~\eqref{eq:qs-pre} becomes the quasi-stability inequality
\begin{equation}
\label{eq:qs}
    \norm{S(t^*)U_0-S(t^*)U_1}_{\mathcal{H}}
    \leq q\,\norm{U_0-U_1}_{\mathcal{H}}
    + n_{\mathcal{H}}(U_0-U_1), \qquad q<1.
\end{equation}

It remains to show that $n_{\mathcal{H}}$ is a compact seminorm.
Let $\{V_n\}\subset\mathcal{H}$ be a bounded sequence with
$\norm{V_n}_{\mathcal{H}}\leq R$. We apply the
Arzel\`{a}--Ascoli theorem to the family
$\{t\mapsto\mathfrak{B}S(t)V_n\}_{n\geq 1}$ in
$C([0,t^*];\mathcal{H})$, verifying the three required
conditions.

\smallskip

\noindent\emph{Uniform boundedness.}
Since $S(t)$ maps bounded sets to bounded sets and $\mathfrak{B}$
is bounded:
\[
    \sup_{t\leq t^*}\norm{\mathfrak{B}S(t)V_n}_{\mathcal{H}}
    \leq \norm{\mathfrak{B}}_{\mathcal{L}(\mathcal{H})}\,CR.
\]

\noindent\emph{Uniform equicontinuity.}
By hypothesis~\eqref{eq:K}, the map $t\mapsto\mathfrak{B}S(t)V$
is differentiable in $\mathcal{H}$ with derivative bounded by
$K$ uniformly in $t\geq 0$ and $V\in\mathbb{B}$.
The fundamental theorem of calculus gives, for
$s,t\in[0,t^*]$:
\begin{equation}
\label{eq:equicont}
    \norm{\mathfrak{B}S(t)V_n-\mathfrak{B}S(s)V_n}_{\mathcal{H}}
    \leq \int_s^t
      \left\|\frac{d}{d\tau}\mathfrak{B}S(\tau)V_n
      \right\|_{\mathcal{H}}d\tau
    \leq K\,|t-s|,
\end{equation}
with $K$ independent of $n$.

\noindent\emph{Pointwise precompactness.}
For each fixed $t\in[0,t^*]$, the sequence $\{S(t)V_n\}$ is
bounded in $\mathcal{H}$, so $\{\mathfrak{B}S(t)V_n\}$ is
precompact in $\mathcal{H}$ by compactness of $\mathfrak{B}$.

\smallskip

By the Arzel\`{a}--Ascoli theorem, there exists a subsequence
along which $\mathfrak{B}S(\cdot)V_{n_k}$ converges in
$C([0,t^*];\mathcal{H})$, so that
\[
    \sup_{t\leq t^*}
    \norm{\mathfrak{B}S(t)(V_{n_k}-V_{n_l})}_{\mathcal{H}}
    \to 0.
\]
Hence \( n_{\mathcal{H}} \) is a compact seminorm on \( \mathcal{H} \). This implies that the dynamical system \( (\mathcal{H}, S(t)) \) is quasi-stable. In particular, the global attractor \( \mathfrak{A} \) has finite fractal dimension (see \cite[Theorem~3.4.5]{Chueshov2015}).

\medskip
\noindent\textbf{Proof of (ii).}
We apply \cite[Theorem~3.4.7]{Chueshov2015}. In addition to the quasi-stability inequality, this result requires the Lipschitz-in-time estimate \cite[(3.4.4)--(3.4.5)]{Chueshov2015}. We establish this estimate in the space \(\mathcal{H}_{-1}\).
Indeed, performing the change of variables \(\sigma \mapsto t-\sigma\) in the Duhamel formula~\eqref{eq:Duhamel}, we obtain
\begin{equation}
\label{eq:Duhamel-cv}
U(t) = T(t)U_0 + \int_0^t T(\sigma)\mathcal{F}(U(t-\sigma))\,d\sigma.
\end{equation}

For $0\leq s\leq t$, subtracting~\eqref{eq:Duhamel-cv}
evaluated at $s$ from~\eqref{eq:Duhamel-cv} evaluated at $t$,
and splitting the resulting integral over $[0,s]$ and $[s,t]$:
\begin{equation}
\label{eq:decomp}
U(t)-U(s)
= \underbrace{(T(t)-T(s))U_0}_{J_1}
+ \underbrace{\int_0^s T(\sigma)
    \bigl[\mathcal{F}(U(t-\sigma))
          -\mathcal{F}(U(s-\sigma))\bigr]
    d\sigma}_{J_2}
+ \underbrace{\int_s^t T(\sigma)
    \mathcal{F}(U(t-\sigma))\,d\sigma}_{J_3}.
\end{equation}

\noindent\emph{Estimate of $J_3$.}
Since $\norm{T(\sigma)}\leq Me^{-\omega\sigma}\leq M$
and $\norm{\mathcal{F}(U(t-\sigma))}_{\mathcal{H}}\leq C_F$:
\begin{equation}
\label{eq:J3}
    \norm{J_3}_{\mathcal{H}} \leq MC_F\,|t-s|.
\end{equation}

\noindent\emph{Estimate of $J_2$.}
By Remark~\ref{rem:H2-consequence}, we obtain 
\begin{equation}
\label{eq:FU-diff}
    \norm{\mathcal{F}(U(t-\sigma))
          -\mathcal{F}(U(s-\sigma))}_{\mathcal{H}}
    \leq \int_{s-\sigma}^{t-\sigma}
         \left\|\frac{d}{d\tau}\mathcal{F}(U(\tau))
         \right\|_{\mathcal{H}}d\tau
    \leq \widetilde{K}\,|t-s|.
\end{equation}
holds with constant $\widetilde{K}$. For every
$\sigma\in[0,s]$:
Using $\norm{T(\sigma)}\leq Me^{-\omega\sigma}$:
\begin{equation}
\label{eq:J2}
    \norm{J_2}_{\mathcal{H}}
    \leq \frac{M\widetilde{K}}{\omega}\,|t-s|.
\end{equation}

\noindent\emph{Estimate of $J_1$ in $\mathcal{H}_{-1}$.}
Since $\mathbb{A}^{-1}$ commutes with $T(t)$:
\begin{equation}
\label{eq:J1-norm}
    \norm{J_1}_{-1}
    = \norm{\mathbb{A}^{-1}(T(t)-T(s))U_0}_{\mathcal{H}}
    = \norm{(T(t)-T(s))\mathbb{A}^{-1}U_0}_{\mathcal{H}}.
\end{equation}
Since $\mathbb{A}^{-1}U_0\in D(\mathbb{A})$ for every
$U_0\in\mathcal{H}$, the fundamental theorem of calculus for
semigroups and the identity $\mathbb{A}\mathbb{A}^{-1}=I$ give:
\begin{equation}
\label{eq:J1-bound}
    \norm{J_1}_{-1}
    = \left\|\int_s^t T(\tau)\mathbb{A}\mathbb{A}^{-1}U_0
       \,d\tau\right\|_{\mathcal{H}}
    = \left\|\int_s^t T(\tau)U_0\,d\tau\right\|_{\mathcal{H}}
    \leq M\norm{U_0}_{\mathcal{H}}\,|t-s|.
\end{equation}

\noindent
Finally, 
since $\norm{\cdot}_{-1}\leq\norm{\cdot}_{\mathcal{H}}$, the
estimates for $J_2$ and $J_3$ also hold in $\mathcal{H}_{-1}$.
Combining~\eqref{eq:J3}, \eqref{eq:J2},
and~\eqref{eq:J1-bound}:
\begin{equation}
\label{eq:lip-time}
    \norm{U(t)-U(s)}_{-1}
    \leq C\,|t-s|, \qquad
    C := M\norm{U_0}_{\mathcal{H}}
         +\frac{M\widetilde{K}}{\omega}
         +MC_F < \infty.
\end{equation}

\noindent
Next we show the Lipschitz continuity in $\mathcal{H}$.
That is
we verify condition~(3.4.4) of \cite[Theorem~3.4.7]
{Chueshov2015}. From Theorem~\ref{LLP1} and
Corollary~\ref{CoLLP1}, solutions remain bounded when the
initial data lie in a bounded ball. Applying the Duhamel
formula~\eqref{eq:Duhamel} to the difference of two
trajectories and using the local Lipschitz continuity of
$\mathcal{F}$ together with Gronwall's inequality, we obtain
\begin{equation}
\label{eq:lip-data}
    \norm{S(t)U_0-S(t)U_1}_{\mathcal{H}}
    \leq C\norm{U_0-U_1}_{\mathcal{H}},
    \quad \forall\,U_0,U_1\in B_R,\quad
    \forall\,t\in[0,t_0].
\end{equation}

Since both conditions of \cite[Theorem~3.4.7]{Chueshov2015}
are satisfied by~\eqref{eq:lip-time} and~\eqref{eq:lip-data},
item~(ii) follows.
\end{proof}

\begin{rem}
Hypothesis~\rm(H2) is the key structural condition on the
nonlinearity. The factorisation $\mathcal{F}(U)=\mathcal{F}(\mathfrak{B}U)$
through a compact operator $\mathfrak{B}$ means that $\mathcal{F}$
depends on $U$ only through its image under $\mathfrak{B}$,
which takes values in a precompact set. This is the source of
the compact seminorm in the quasi-stability inequality and is
what ultimately gives the fractal exponential attractor finite
dimension in $\mathcal{H}_{-1}$.
\end{rem}

\begin{rem}
The fractal dimension of $\mathfrak{M}$ is finite in the
extrapolation space $\mathcal{H}_{-1}$, not necessarily in the phase space $\mathcal{H}$.
Since $\mathcal{H}\hookrightarrow\mathcal{H}_{-1}$ is a continuous and dense
embedding and $\mathfrak{A}\subset\mathfrak{M}$, the global
attractor $\mathfrak{A}$ also has finite fractal dimension in
$\mathcal{H}_{-1}$:
\begin{equation}
\dim_f^{\mathcal{H}_{-1}}\mathfrak{A}
\leq \dim_f^{\mathcal{H}_{-1}}\mathfrak{M} < \infty.
\end{equation}
\end{rem}

\section{Applications to the Semilinear Model}\label{section-4}
\setcounter{equation}{0}

In this section we apply the abstract framework developed above to the
semilinear suspension bridge model
\begin{align}
\rho_1 v_{tt} - \beta_0 v_{xx} + k(v-u)
  &= F(v,w),
  \qquad (x,t)\in I\times\mathbb{R}^+_0,
  \label{eq:cableLN}
\\
\rho_2 u_{tt} + \alpha u_{xxxx} - \alpha_0 u_{xx} + k(u-v)
  &= G(v,w),
  \qquad (x,t)\in I\times\mathbb{R}^+_0,
  \label{eq:deckNL}
\end{align}
supplemented with the transmission conditions~\eqref{ct.2}.

Setting $\mathcal{F}(U):=(0,F,0,G)^\top$ and $\mathbb{A}:=\mathcal{A}$
in~\eqref{absCE}, where $\mathcal{A}$ is the operator defined
in~\eqref{AA}, the system reduces to the abstract semilinear
evolution equation
\begin{equation}
U_t = \mathcal{A}U + \mathcal{F}(U).
\end{equation}
By Theorem~\ref{LLP1} and Corollary~\ref{CoLLP1}, global existence
and the asymptotic properties of solutions follow whenever
$\mathcal{F}$ satisfies assumptions~\eqref{ff2}
and~\eqref{Freg1}--\eqref{Freg2}. We now verify these assumptions
for three classes of nonlinearities.

\subsection*{First example: separated power-type nonlinearities}

Consider nonlinearities of the separated form
\begin{equation}\label{FG1}
F(u,v) = F(u) = \mu_1 u|u|^{\alpha},
\qquad
G(u,v) = G(v) = \mu_2 v|v|^{\beta}.
\end{equation}
By the mean-value theorem,
\begin{equation}
\bigl|s|s|^\alpha - r|r|^\alpha\bigr|
\leq
\bigl(|s|^\alpha + |r|^\alpha\bigr)|s-r|,
\end{equation}
so $F$ and $G$ are locally Lipschitz. To obtain globally Lipschitz
approximations, we introduce, for $R>0$, the truncated nonlinearity
\begin{equation}\label{eq:trunc1}
F_{1,R}(s) =
\begin{cases}
\mu_1 s|s|^\alpha & |s| \leq R,\\[4pt]
\mu_1 R^\alpha s  & |s| > R,
\end{cases}
\end{equation}
and define $G_{1,R}$ analogously. Each truncation is globally Lipschitz
with constant $\mu_i R^\alpha$ (resp.\ $\mu_i R^\beta$), and
satisfies assumptions~\eqref{ff2} and~\eqref{Freg1}--\eqref{Freg2}.

\subsection*{Second example: gradient nonlinearities}

We now consider nonlinearities arising from a gradient vector field,
\begin{equation}\label{FG2}
(F,G) = \nabla p,
\end{equation}
with a polynomial potential
\begin{equation}
p(x,y) = x^{2m}y^{2n},
\qquad m,n\in\mathbb{N}.
\end{equation}
Note that $p\geq 0$ everywhere. Set $N:= 2m+2n$ and
$r:=\sqrt{x^2+y^2}$.

Since $F(x,y) = 2mx^{2m-1}y^{2n}$ and $G(x,y) = 2nx^{2m}y^{2n-1}$ are polynomials and hence only locally Lipschitz, a global Lipschitz approximation requires a more delicate construction. 
The key requirement is that the truncated potential $p_R$ belongs to $C^2(\mathbb{R}^2)$, so that $D^2 p_R\in L^\infty(\mathbb{R}^2)$,
and the mean-value theorem yields a global Lipschitz bound on
$\nabla p_R$.

To achieve $C^2$ regularity across the boundary $\partial B_R$,
we introduce the quintic Hermite transition polynomial
\begin{equation}
h(s) = 1 - 10s^3 + 15s^4 - 6s^5,
\qquad 0\leq s\leq 1,
\end{equation}
which satisfies the boundary conditions
\begin{equation}
h(0)=1,\quad h'(0)=0,\quad h''(0)=0,
\qquad
h(1)=0,\quad h'(1)=0,\quad h''(1)=0.
\end{equation}
These conditions ensure that the radial weight function
\begin{equation}\label{eq:phi_R}
\varphi_R(r) =
\begin{cases}
1
  & 0 \leq r \leq R,
\\[6pt]
h\!\left(\dfrac{r-R}{R}\right)
+
\left(1 - h\!\left(\dfrac{r-R}{R}\right)\right)
\!\left(\dfrac{R}{r}\right)^{N-1}
  & R < r < 2R,
\\[10pt]
\left(\dfrac{R}{r}\right)^{N-1}
  & r \geq 2R,
\end{cases}
\end{equation}
belongs to $C^2([0,\infty))$. Indeed, a direct computation gives
\begin{equation}
\varphi_R(R)=1,\quad
\varphi_R'(R)=0,\quad
\varphi_R''(R)=0,
\end{equation}
and
\begin{equation}
\varphi_R(2R)=\frac{1}{2^{N-1}},\quad
\varphi_R'(2R)=-\frac{N-1}{2^N R},\quad
\varphi_R''(2R)=\frac{N(N-1)}{2^{N+1}R^2},
\end{equation}
confirming that both the function and its first two derivatives
match continuously at $r=R$ and $r=2R$.
We define the truncated potential
\begin{equation}
p_R(x,y) := \varphi_R(r)\,x^{2m}y^{2n}
\end{equation}
and the corresponding truncated vector field
\begin{equation}
(F_R,G_R) := \nabla p_R.
\end{equation}
Since $\partial_x r = x/r$ and $\partial_y r = y/r$, explicit
differentiation yields the relations
\begin{align*}
F_R(x,y)
  &= \partial_x p_R(x,y)
   = 2m\,\varphi_R(r)\,x^{2m-1}y^{2n}
   + \varphi_R'(r)\,\frac{x^{2m+1}y^{2n}}{r}
\intertext{and}
G_R(x,y)
  &= \partial_y p_R(x,y)
   = 2n\,\varphi_R(r)\,x^{2m}y^{2n-1}
   + \varphi_R'(r)\,\frac{x^{2m}y^{2n+1}}{r}.
\end{align*}
Since
\begin{equation*}
r^{N-1}\varphi_R(r),\quad 
r^N\varphi_R'(r)\qtq{and}
r^{N+1}\varphi_R''(r)
\end{equation*}
are bounded for $r\ge R$ and hence for $r\ge 0$ as well,
the Hessian $D^2p_R$ of $p_R$ is bounded, 
and the mean-value theorem yields
\begin{equation}
|\nabla p_R(X) - \nabla p_R(Y)|
\leq
\|D^2p_R\|_{L^\infty(\mathbb{R}^2)}\,|X-Y|
\qtq{for all} X,Y\in\mathbb{R}^2.
\end{equation}
Consequently,
\begin{equation}
|(F_R,G_R)(X) - (F_R,G_R)(Y)| \leq C_R\,|X-Y|,
\end{equation}
so that the truncated nonlinearities are globally Lipschitz and satisfy the assumptions~\eqref{ff2} and~\eqref{Freg1}--\eqref{Freg2}.

\begin{rem}
The role of $h$ is to interpolate $C^2$-smoothly between
$\varphi_R\equiv 1$ (inside $B_R$, where no truncation occurs)
and the power-decay $\varphi_R(r)=(R/r)^{N-1}$ (outside $B_{2R}$,
which ensures that $p_R$ decays at infinity). 
\end{rem}

\subsection*{Third example: asymmetric nonlinearity}

The typical contact nonlinearities
\begin{equation}\label{FG3}
F(v-u) = -k(v-u)^-,
\qquad
G(u-v) = -k(u-v)^-,
\end{equation}
where $s^- := \min(s,0)$ denotes the negative part, are already
globally Lipschitz with constant $k$. The abstract results therefore
apply directly, without any truncation.

\subsection*{Inhomogeneous nonlinearities}

In all three examples above, one may add a fixed forcing term
$F\in\mathcal{H}$ and consider the perturbed nonlinearity
\begin{equation}
\widetilde{\mathcal{F}}(U) := \mathcal{F}(U) + F.
\end{equation}
Since translation by a constant element of $\mathcal{H}$ preserves
local and global Lipschitz continuity, as well as
assumptions~\eqref{ff2} and~\eqref{Freg1}--\eqref{Freg2},
all previous conclusions remain valid for $\widetilde{\mathcal{F}}$.

\noindent
Assuming that $0\in\varrho(\mathcal{A})$, the operator $\mathcal{A}^{-1}$ is
well defined and bounded on $\mathcal{H}$. We introduce the norm
\begin{equation}
    \|U\|_{-1} := \|\mathcal{A}^{-1}U\|_{\mathcal{H}}
\end{equation}
and define the extrapolation space
$\mathcal{H}_{-1}:=\overline{\mathcal{H}}^{\,\|\cdot\|_{-1}}$.
It is not difficult to see that 
\begin{equation}
\label{eq:H-1}
\mathcal{H}_{-1}
= L^2(0,\ell)\times\mathcal{V}_1^*\times L^2(0,\ell)\times\mathcal{V}_2^*,
\end{equation}

\noindent 
where  $\mathcal{V}_1^*$ and $\mathcal{V}_2^*$ are given by 
\begin{align}
\label{eq:V1star}
\mathcal{V}_1^*
&= \bigl\{f\in H^{-1}(0,\ell)\ :
   \ f=-\alpha_0 v_{xx}+k(v-u)\ \text{for some }
   v\in\mathcal{V}_1\bigr\},
\\[4pt]
\label{eq:V2star}
\mathcal{V}_2^*
&= \bigl\{f\in H^{-2}(0,\ell)\ :
   \ f=\alpha u_{xxxx}-\alpha_0 u_{xx}+k(u-v)\ \text{for some }
   u\in\mathcal{V}_2\bigr\}.
\end{align}
Therefore the completion of $\mathcal{H}$ under $\|\cdot\|_{-1}$ is
endowed with the norm
\begin{equation}
\label{eq:H-1-norm}
\|(f_1,f_2,f_3,f_4)\|_{\mathcal{H}_{-1}}^2
= \frac{1}{2}\Bigl(
  \|f_1\|_{L^2}^2
  +\|f_2\|_{\mathcal{V}_1^*}^2
  +\|f_3\|_{L^2}^2
  +\|f_4\|_{\mathcal{V}_2^*}^2
  \Bigr).
\end{equation}

\begin{Th}
\label{Attractor2}
Let $S(t)$ be the semigroup associated with the semilinear
suspension bridge system~\eqref{eq:cableLN}--\eqref{eq:deckNL}.
Then the following hold.
\begin{enumerate}
\item[\rm(i)] $S(t)$ possesses a unique compact global attractor
$\mathfrak{A}\subset D(\mathcal{A})$.

\item[\rm(ii)] $S(t)$ possesses a fractal exponential attractor
$\mathfrak{M}\subset\mathcal{H}_{-1}$ satisfying:
\begin{enumerate}
\item[\rm(a)] $\mathfrak{A}\subset\mathfrak{M}$ as subsets
      of $\mathcal{H}_{-1}$;
\item[\rm(b)] $\mathfrak{M}$ is positively invariant under
      $S(t)$;
\item[\rm(c)] $\mathfrak{M}$ attracts every bounded set
     $B\subset\mathcal{H}$ exponentially in $\mathcal{H}_{-1}$:
there exist constants $C_B,\gamma>0$ such that
\begin{equation}
\mathrm{dist}_{\mathcal{H}_{-1}}
\bigl(S(t)B,\,\mathfrak{M}\bigr)
\leq C_B\,e^{-\gamma t},
\qquad t\geq 0;
\end{equation}
\item[\rm(d)] $\mathfrak{M}$ has finite fractal dimension in
 $\mathcal{H}_{-1}$:
\begin{equation}
\dim_f^{\mathcal{H}_{-1}}\mathfrak{M}<\infty.
\end{equation}
\end{enumerate}
In particular,
$\dim_f^{\mathcal{H}_{-1}}\mathfrak{A}
\leq\dim_f^{\mathcal{H}_{-1}}\mathfrak{M}<\infty$.
\end{enumerate}
\end{Th}

\begin{proof}
We Check the hypotheses of Theorem~\ref{Attractor}.
The operator $\mathcal{A}$ defined in~\eqref{AA} has a compact resolvent because the embedding $D(\mathcal{A})\hookrightarrow\mathcal{H}$ is compact. 
This follows from the compact Sobolev embeddings
$H^2(0,\ell)\hookrightarrow L^2(0,\ell)$ and
$H^4(0,\ell)\hookrightarrow H^2(0,\ell)$, together with the
structure of $D(\mathcal{A})$ given in~\eqref{domA}. 
Hence hypothesis (H1) holds. 

\medskip
The nonlinear mapping $\mathcal{F}$ defined by the
nonlinearities~\eqref{FG1}--\eqref{FG3} satisfies (H2)
 and (H3) follows from the fact that
$\mathcal{F}$ depends on $U$ through a compact operator
$\mathfrak{B}:\mathcal{H}\to\mathcal{H}$ defined in Remark \ref{RBBB}.

\medskip
\noindent 
Since all hypotheses of Theorem~\ref{Attractor}
are satisfied, conclusions~(i) and~(ii) follow directly.
The inclusion $\mathfrak{A}\subset\mathfrak{M}$ in
$\mathcal{H}_{-1}$ holds because $\mathfrak{A}$ is bounded in
$\mathcal{H}$, hence
$\mathrm{dist}_{\mathcal{H}_{-1}}(S(t)\mathfrak{A},\mathfrak{M})
\leq C_{\mathfrak{A}}e^{-\gamma t}\to 0$,
and $\mathfrak{M}$ is closed in $\mathcal{H}_{-1}$.
\end{proof}

\subsection*{Acknowledgements}
The first author has been partially supported by the National Natural Science Foundation of China (NSFC) \#12571098.
The second author has been  supported by ANID Fondecyt 1230914, Chile and   CNPq Grant 307947/2022-0 Brazil. 
\\

\subsection*{Declarations}
{\bf Conflict of interest:}
The authors of this paper declare no conflict of interest.

\noindent{\bf Ethical approval:}
This article does not contain any studies with human participants or animals performed by
any of the authors.

\end{document}